\documentclass[makeidx]{article}
\title{Eigenvalue bounds for the magnetic Laplacian}
\author{Bruno Colbois and Alessandro Savo}
\date{\today}
\usepackage{makeidx}
\usepackage{amssymb , amsmath, amsthm}
\usepackage{graphicx}
\newtheorem{defi}{Definition} 
\newtheorem{thm}[defi]{Theorem}
\newtheorem{ex}[defi]{Example}
\newtheorem{rem}[defi]{Remark}
 \newtheorem{prop}[defi]{Proposition}
\newtheorem{lemme}[defi]{Lemma}

\newcommand{\twosystem}[2]{\left\{\begin{aligned} &#1\\ &#2\end{aligned}\right.}
\newcommand{\threesystem}[3]{\left\{ \begin{aligned}&#1\\ &#2\\&#3\end{aligned}\right.}

\newcommand{\nero}{\smallskip$\bullet\quad$\rm}

\newcommand{\due}[2]{#1&#2\\}

\newcommand{\matrice}{\begin{pmatrix}}
\newcommand{\ok}{\end{pmatrix}}
\newcommand{\twomatrix}[4]{\matrice\due {#1}{#2}\due{#3}{#4}\ok}

\newcommand{\derive}[2]{\dfrac{\bd #1}{\bd#2}}

\newcommand{\opd}[1]{\dfrac{\bd}{\bd #1}}

\newcommand{\Due}[2]{\begin{pmatrix}#1\\#2\end{pmatrix}}

\newcommand{\scal}[2]{\langle{#1},{#2}\rangle}

\newcommand{\abs}[1]{\lvert{#1}\rvert}
\newcommand{\norm}[1]{\lVert{#1}\rVert}

\newcommand{\reals}{{\bf R}}
\newcommand{\sphere}[1]{{\bf S}^{#1}}
\newcommand{\real}[1]{{\bf R}^{#1}}

\newcommand{\bd}{\partial}

\renewcommand{\l}{\lambda}

\begin{document}

\maketitle
\begin{abstract} 
 We consider a compact Riemannian manifold $M$ endowed with a potential $1$-form $A$ and study the magnetic Laplacian $\Delta_A$ associated with those data (with Neumann magnetic boundary condition if $\partial M \not = \emptyset$). We first establish a family of upper bounds for all the eigenvalues, compatible with the Weyl law. When the potential is  a closed $1$-form, we get a sharp upper bound for the first eigenvalue. In the second part, we consider only closed potentials, and we establish a sharp lower bound for the first eigenvalue when the manifold is a 2-dimensional Riemannian cylinder. The equality case characterizes the situation where the metric is a product. We also look at the case of doubly convex domains in the Euclidean plane.

\medskip



\end{abstract}
\large

\section{Introduction} \label{intro} Let $(M,g)$ be a complete Riemannian manifold and let $\Omega$ be a compact domain of $M$ with smooth boundary $\bd\Omega$,  if non empty (if $M$ is closed, we can choose $\Omega=M$ and $\partial \Omega=\emptyset$).
Consider the trivial complex line bundle $\Omega\times\bf C$ over $\Omega$; its space of sections  can be identified with $C^{\infty}(\Omega,\bf C)$, the space of smooth complex valued functions on $\Omega$.  Given a smooth real 1-form $A$ on $\Omega$ we define a connection $\nabla^A$ on $C^{\infty}(\Omega,\bf C)$ as follows:
\begin{equation} \label{connection}
\nabla^A_Xu=\nabla_Xu-iA(X)u
\end{equation}
for all vector fields $X$ on $\Omega$ and for all $u\in C^{\infty}(\Omega,\bf C)$. The operator
\begin{equation} \label{magnetic laplacian}
\Delta_A=(\nabla^A)^{\star}\nabla ^A
\end{equation}
is called the {\it magnetic Laplacian} associated to the magnetic potential $A$, and the smooth two form 
$$
B=dA
$$
is the  associated {\it  magnetic field}. If the boundary of $\Omega$ is non empty, we will consider Neumann magnetic conditions, that is:
\begin{equation}\label{mneumann}
\nabla^A_Nu=0\quad\text{on}\quad\bd\Omega,
\end{equation}
where $N$ denotes the inner unit normal. 
Then, it is well-known that $\Delta_A$ is self-adjoint, and admits a discrete spectrum
$$
0\le \l_1(\Delta_A)\le \l_2(\Delta_A) \le ... \to \infty.
$$

The above  is a particular case of a more general situation, where $E\to M$ is a complex line bundle with a hermitian connection $\nabla^E$, and where the magnetic Laplacian is defined as $\Delta_E=(\nabla^E)^{\star}\nabla ^E$.

\smallskip

The spectrum of the magnetic Laplacian is very much studied in analysis (see for example [BDP] and the references therein) and in relation with physics; lower estimates of its fundamental tone have been worked out, in particular, when $\Omega$ is a planar domain and $B$ is the constant magnetic field; that is, when the function $\star B$ is constant on $\Omega$ (see for example a Faber-Krahn type inequality in [Er1] and the recent[LS] and the references therein). The case when the potential $A$ is a closed $1$-form is particularly interesting from  the physical point of view (Aharonov-Bohm  effect), and also from the geometric point of view: this situation corresponds to {\it zero} magnetic field.
We refer in particular to the paper [HHHO], where the authors study the multiplicity and the nodal sets corresponding to the ground state $\l_1$ for non-simply connected planar domains with harmonic potential and Neumann boundary condition. For holed plane domains and Dirichlet boundary condition, there is a serie of papers for domains with a pole, when the pole approaches the boundary (see [AFNN], [NNT] and the references therein). Last but not least, there is a Aharonov-Bohm approach to the question of  nodal and minimal partitions, see chapter 8 of [BH].

\smallskip
In this paper, we are interested in  a more geometric approach to the spectrum of $\Delta_A$, in the same spirit of what is done for the usual Laplacian and Laplace type operators. In this context, let us mention the recent article [LLPP] (chapter 7) where the authors establish a \emph{Cheeger type inequality} for $\l_1$; that is, they find a lower bound for $\l_1(\Delta_A)$ in terms of the geometry of $\Omega$ and the potential $A$. In the preprint [ELMP], the authors approach the problem via the Bochner method.

In a more general context, in [BBC], the authors establish a lower bound for $\l_1(\Delta_A)$ in terms of the \emph{holonomy} of the vector bundle on which $\Delta_A$ acts. In both cases, implicitly, the harmonic part of the potential $A$ (or, more precisely, the distance between the harmonic part of $A$ and the lattice of integral 1-harmonic forms) plays a crucial role.
 
\smallskip
Our goal in this article is, firstly,  to give a very general family of \emph{upper bounds} for the whole spectrum $\{\lambda_k(\Delta_A)\}$ of $\Delta_A$ in term of the geometry of $\Omega$ and  the potential $A$: see Theorem \ref{main1} below. We stress that our estimates are optimal, in the sense that they are compatible with the Weyl law. Moreover, these estimates show how the geometry of the manifold and the data coming from the potential interplay.

\smallskip
In the case when $A$ is closed  (i.e. zero magnetic field), we establish a sharp upper bound for the first eigenvalue $\lambda_1(\Delta_A)$ in term of the distance of $A$ to the an integral  lattice of harmonic $1$-forms (see Proposition \ref{harmonic potential}). We also get an upper bound by comparison of the spectrum of $\Delta_A$ with the spectrum of the Laplacian on domains $D \subset \Omega$ with mixed Dirichlet/Neumann boundary conditions.

\smallskip
In the second part, inspired by the paper [HHHO], we continue to study the particular the case of a closed potential, where it is possible to get \emph{lower bounds} for $\lambda_1(\Delta_A)$ in the specific but fundamental situation where $\Omega$ is a Riemannian cylinder of dimension $2$. We establish a sharp lower bound in terms of the geometry of $\Omega$ and the \emph{flux} of $A$, and characterize the equality case: the cylinder has to be a Riemannian product (see Theorem \ref{main3}).
We also study the situation of a convex subset of $\bf R^2$ with a convex hole. Here, we establish a lower bound in terms of the geometry of this doubly convex set and the flux of $A$ (see Theorem \ref{main2} below).  

\medskip

Here is the scheme of this introduction: in Section \ref{preliminary} we  recall some classical facts about the magnetic Laplacian and the Hodge decomposition and,  in Section \ref{zero}, we discuss zero magnetic field and state  Shikegawa's equivalent conditions for the vanishing of the first eigenvalue. In Section \ref{upperbounds} we state our main upper bounds and in Section \ref{lowerbounds} we discuss the lower bounds, first for general Riemannian cylinders and then for doubly connected domains.

%%%%%

\subsection{Preliminary facts and notation} \label{preliminary} First, we recall the variational definition of the spectrum. Let $\Omega$ be a a closed manifold, or a compact domain with smooth boundary in a complete Riemannian manifold $M$. If the boundary is non empty, we assume for $u\in C^{\infty}(\Omega,\bf C)$ the magnetic Neumann conditions, as in \eqref{mneumann}. Then one verifies that
$$
\int_{\Omega}(\Delta_Au)\bar u=\int_{\Omega}\abs{\nabla^Au}^2,
$$
and the associated quadratic form is then
$$
Q_A(u)=\int_{\Omega}\abs{\nabla^Au}^2.
$$
Equivalently, if one defines the modified differential $d^A$ as $d^Au=du-iuA$, then 
$$
Q_A(u)=\int_{\Omega}\abs{d^Au}^2.
$$
This means that the spectrum of $\Delta_A$ admits the usual variational characterization:

\begin{equation}
\lambda_1(\Delta_A)= \min\Big\{ \frac{Q_A(u)}{\Vert u\Vert^2}:\ u\in C^{1}(\Omega,\mathbb C) / \{0\}\Big\}
\end{equation}
and
\begin{equation}
\lambda_k(\Delta_A)= \min_{E_k} \max \Big\{ \frac{Q_A(u)}{\Vert u\Vert^2}:\ u\in E_k  / \{0\}\Big\}
\end{equation}
where $E_k$ runs through the set of all $k$-dimensional vector subspaces of $C^{1}(\Omega,\mathbb C)$. Note that $\lambda_1(\Delta_A)\geq 0$, and we will discuss conditions under which $\lambda_1$ is positive in Theorem \ref{shikegawa} below. 

\smallskip

\nero {\it We will also use the notation
$$
\lambda_k(\Omega,A)\doteq\lambda_k(\Delta_A)
$$
when we want to stress the manifold $\Omega$ on which the Laplacian is acting. }

\medskip
The following proposition recalls some well-known facts. We say that a differential form $A$ is {\it tangential} if $i_NA=0$ on $\bd\Omega$.

\begin{prop}  \label{basic facts}
\begin{enumerate}
\item
The spectrum of $\Delta_A$ is equal to the spectrum of $\Delta_{A+d\phi}$ for all smooth real valued functions $\phi$; in particular, when $A$ is exact, the spectrum of $\Delta_A$ reduces to that of the classical Laplace-Beltrami operator acting on functions (with Neumann boundary conditions if $\bd\Omega$ is not empty).

\item 
 Let $A$ be $1-$form on $\Omega$. Then, there exists a smooth real valued function $\phi$ on $\Omega$ such that the $1$-form $\tilde A=A+d\phi$ is co-closed and tangential, that is:
\begin{equation} \label{boundary conditions}
\delta\tilde A=0, \,\, i_N\tilde A=0.
\end{equation}

\item 
 Set:
$$
{\rm Har}_1(\Omega)=\Big\{h\in\Lambda^1(\Omega): dh=\delta h=0 \,\,\text{on $\Omega$},\,\, i_Nh=0\,\text{on}\,\,\bd\Omega\Big\}.
$$
Assume that the $1$-form $A$ is co-closed and tangential. Then $A$ can be decomposed
\begin{equation} \label{decomposition}
A=\delta\psi+h,
\end{equation}
where $\psi$ is a smooth tangential $2$-form and $h\in{\rm Har}_1(\Omega)$. Note that the vector space ${\rm Har}_1(\Omega)$ is isomorphic to the first de Rham absolute cohomology space $H^1(\Omega,\reals)$. 

\end{enumerate}
\end{prop}

\nero {\rm Assertion (1) expresses the well-known {\it Gauge invariance} of the spectrum. Thanks to Assertion (2), in the study of the spectrum of the magnetic Laplacian, we can always assume that the potential $A$ is co-closed and tangential.}  

\begin{proof} 
\begin{enumerate}

\item 
 This comes from the fact that
\begin{equation}\label{gauge}
\Delta_A e^{-i\phi}=e^{-i\phi} \Delta_{A+d\phi}
\end{equation}
hence $\Delta_A$ and $\Delta_{A+d\phi}$ are unitarily equivalent.

\item 
Observe that the problem:
$$
\twosystem
{\Delta\phi=-\delta A \quad\text{on}\quad \Omega,}
{\frac{\partial\phi}{\partial N}=-A(N)\quad\text{on}\quad\partial \Omega}
$$
has a unique solution (modulo a multiplicative constant). It is immediate to verify that $\tilde A=A+d\phi$ is indeed co-closed and tangential.

\item
 We apply the Hodge decomposition to the $1$-form $A$ (see [Sc], Thm. 2.4.2), and get:
\begin{equation}
A=df + \delta \psi +h,
\end{equation}
where $f$ is a function which is zero on the boundary, $\psi$ a tangential 2-form and $h$ is a 1-form satisfying $dh=\delta h=0$ (in particular, $h$ is harmonic). Now, as $\delta A=0$ we obtain $\delta df=0$ hence $f$ is a harmonic function; since $f$ is zero on the boundary, we get $f=0$ also on $\Omega$ and we can write
\begin{equation}
A=\delta \psi +h.
\end{equation}
Now, since both $A$ and $\delta\psi$ are tangential, also $h$ will be tangential. 
\end{enumerate}
\end{proof}

\nero {\it In the sequel, when we write the decomposition $A=\delta \psi +h$, it will be implicitely supposed that $\psi$ a tangential 2-form and $h$ is a 1-form satisfying $dh=\delta h=0$ and $i_Nh=0$.}

%%%

\subsection{Zero magnetic field}\label{zero} We now focus on the first eigenvalue. 
Clearly, if $A=0$,  then $\lambda_1(\Delta_A)=0$ simply because $\Delta_A$ reduces to the usual Laplacian, which has first eigenvalue equal to zero and first eigenspace spanned by the constant functions. If $A$ is exact, then $\Delta_{A}$ is unitarily equivalent to $\Delta$, hence, again, $\lambda_1(\Delta_A)=0$. In fact one checks easily from the definition of the connection  that, if $A=d\phi$ for some real-valued function $\phi$ then:
$$
\nabla^{A}e^{i\phi}=0,
$$
that is, $u=e^{i\phi}$ is $\nabla^A$-parallel hence $\Delta_A$-harmonic. 

On the other hand, if the magnetic field $B=dA$ is non-zero then $\lambda_1(\Delta_A)>0$ (see Theorem \ref{shikegawa} below).  It then remains to examine the case when $A$ is closed but not exact. 
The situation was clarified by Shikegawa in [Sh]. First, let us recall that, if $A$ is a $1$-form and $c$ is a closed curve (a loop), the quantity:
$$
\Phi^A_c=\dfrac{1}{2\pi}\oint_{c}A
$$
is called the {\it flux} of $A$ across $c$. 

\nero {\it We will not  specify the orientation of the loop, so that the flux will only be defined up to sign. This will not affect any of the statements, definitions or results which we will prove in this paper.} 

\smallskip

Then we have (see [Sh]):

\begin{thm}\label{shikegawa}The following statements are equivalent:

\item a)
$\lambda_1(\Delta_A)=0$;

\item b)
$dA=0$ and $\Phi^A_c\in\bf Z$ for any closed curve $c$ in $\Omega$.
\end{thm}

Thus, the first eigenvalue vanishes if and only if $A$ is a closed form whose flux around every closed curve is an integer; equivalently, if $A$ has non-integral flux around at least one closed loop, then $\lambda_1(\Delta_A)>0$. 

Note that, by applying part (2) of Proposition \ref{basic facts} in the study of the spectrum, any closed potential $A$ can be supposed to be also co-closed and tangential, that is, also an element of ${\rm Har}_1(\Omega)$: of course, the magnetic field $B=dA$ vanishes identically in that case. 

\smallskip

In [HHHO] this situation is studied when $\Omega$ is a planar domain bounded by the outer curve $\Sigma_2$ and the inner curve $\Sigma_1$, and the potential is a harmonic $1$-form in ${\rm Har}_1(\Omega)$.  By analyzing the nodal set of a first eigenfunction, the authors proved that  $\lambda_1(\Delta_A)$ is maximized precisely when the flux of $A$ across $\Sigma_1$ (or, which is the same, across $\Sigma_2$) is congruent to $\frac 12$ modulo integers. The remarkable fact is that this property holds regardless of the geometry of the annulus. It is fair to say that the lower bounds we derive in the second part of the paper were inspired by the paper [HHHO], and in particular we will obtain an explicit lower bound of $\lambda_1(\Delta_A)$ for planar annuli bounded by two convex curves, in terms of the geometry of the annulus and in terms of the distance of the flux of $A$ to the lattice of integers (see (\ref{convex})).  This result is a corollary of a more general lower bound for Riemannian cylinders (see Section \ref{cylinder}). 

\smallskip

Before stating the results, let us define two related distances associated to the 1-form $A$ which will play an important role in our estimates (see \eqref{ldistance} and \eqref{dflux} below). This distances  depend on the choice of a basis for the degree 1 homology of $\Omega$, say $(c_1,\dots,c_m)$. 

\smallskip

Consider the dual basis $(A_1,...,A_m)$ of ${\rm Har}_1(\Omega)$, defined by:
$$
\frac{1}{2 \pi}\oint_{c_i} A_j=\delta_{ij}
$$
and the lattice
$$
{\cal L}_{\bf Z}=\{k_1A_1+\dots+k_mA_m: k_j\in\bf Z\}
$$
which is an abelian subgroup of ${\rm Har}_1(\Omega)$. Given $A\in {\rm Har}_1(\Omega)$, we define its minimum distance to the lattice ${\cal L}_{\bf Z}$ by the formula:
\begin{equation}\label{ldistance}
d(A,{\cal L}_{\bf Z})^2=\min \Big\{\norm{\omega-A}^2,\, \omega\in {\cal L}_{\bf Z}\Big\},
\end{equation}
where $\norm{\cdot}$ denotes the $L^2$-norm of $1$-forms in $\Omega$. 
We stress that this distance is taken in the  Euclidean space ${\rm Har}_1(\Omega)$ endowed with the $L^2$-scalar product of forms. 

\smallskip

On the other hand, given any closed $1$-form $A$, we can consider its {\it flux vector}  in $\real m$ defined by:
$$
\Phi^A=(\Phi^A_{c_1},\dots, \Phi^A_{c_m}),
$$
and denote by $d(\Phi^A,{\bf Z}^m)$ the Euclidean distance of $\Phi^A$ to the lattice ${\bf Z}^m$, that is:
\begin{equation}\label{dflux}
d(\Phi^A,{\bf Z}^m)^2=\min\Big\{\sum_{j=1}^m(\Phi^A_{c_j}-k_j)^2: (k_1,\dots,k_m)\in{\bf Z}^m\Big\}.
\end{equation}
Note that $d(\Phi^A,{\bf Z}^m)\leq \frac {\sqrt m}{2}$. We have the following relation between the two distances.

\begin{prop} Let $A\in{\rm Har}_1(\Omega)$  and let $(A_1,\dots,A_m)$ be the dual basis of a given homology basis $(c_1,\dots,c_m)$. Then
$$
d(A,{\cal L}_{\bf Z})^2\leq  (\norm{A_1}^2+\dots+\norm{A_m}^2)\cdot d(\Phi^A,{\bf Z}^m)^2.
$$
\end{prop}

\begin{proof} By definition, we can write
$$
A=\Phi^A_{c_1}A_1+\dots+\Phi^A_{c_m}A_m.
$$
Let $k_1,\dots,k_m$ be integers. Then:
$$
\begin{aligned}
\norm{A-(k_1A_1+\dots+k_mA_m)}^2&=\norm{(\Phi^A_{c_1}-k_1)A_1+\dots+(\Phi^A_{c_m}-k_m)A_m}^2\\
&\leq \Big((\Phi^A_{c_1}-k_1)^2+\dots+(\Phi^A_{c_m}-k_m)^2\Big)(\norm{A_1}^2+\dots+\norm{A_m}^2)
\end{aligned}
$$
We now take the infimum over all $(k_1,\dots,k_m)\in{\bf Z}^m$ and get the assertion. 
\end{proof}
Note that if the flux across any curve is an integer, then $d(\Phi^A,{\bf Z}^m)=0$ and  also $d(A,{\cal L}_{\bf Z})=0$. On the other hand, the positiveness of any of the two distances guarantees that the first eigenvalue is strictly positive.  
Thus, $d(A,{\cal L}_{\bf Z})$ and $d(\Phi^A,{\bf Z}^m)$ are meaningful invariants for quantitative bounds of $\lambda_1(\Delta_A)$.

\subsection{Upper bounds}\label{upperbounds}

\begin{prop} \label{harmonic potential} Let $A=\delta\psi+h$ be a potential as in (\ref{decomposition})
where $\psi$ is a smooth $2$-form and $h\in{\rm Har}_1(\Omega)$.
Fix a basis for the degree $1$ homology of $\Omega$, and let $(A_1,\dots,A_m)$ be the dual basis of ${\rm Har}_1(\Omega)$.  Let ${\cal L}_{\bf Z}$ be the integral lattice generated by $A_1,\dots,A_m$. Then:

\begin{enumerate}
\item
$$
\l_1(\Delta_A)\leq\frac{ d(h,{\cal L}_{\bf Z})^2}{\vert \Omega\vert}+\frac{\Vert B\Vert^2}{\lambda_{1,1}(\Omega)\vert \Omega\vert}
$$
where $\vert \Omega\vert$ denotes the volume of $\Omega$, $\lambda_{1,1}(\Omega)$ denotes the first eigenvalue of the Laplacian acting on co-exact $1$-forms (with absolute boundary condition if $\partial \Omega \not=\emptyset$), $B=dA$ is the curvature of the magnetic potential $A$ and $\Vert .\Vert$ denotes the $L^2$-norm on $\Omega$.

\item
When $B=0$, we get
$$
\l_1(\Delta_A)\leq\frac{ d(h,{\cal L}_{\bf Z})^2}{\vert \Omega\vert}.
$$
and this inequality is sharp. Equality holds for flat rectangular tori. 
\end{enumerate}
In the above bounds we set $d(h,{\cal L}_{\bf Z})=0$ whenever ${\rm Har}_1(\Omega)=0$.
\end{prop}

This is only the first step in the proof of a quite general upper bound for $\lambda_k(\Delta_A)$, valid for any positive integer $k$.  Let us state it. Assume that $(M^n,g)$ has a Ricci curvature bound:
$$
{\rm Ric}(M,g)\ge -a^2(n-1).
$$
Then we have:

\begin{thm} \label{main1} Let $A$ be any co-closed potential $1$-form on $\Omega$. According to (\ref{boundary conditions}) in Proposition \ref{basic facts}, we can split:
$$
A=h+\delta\psi,
$$
where $dh=\delta h=0$. Then, 
there exist positive constants $c_1,c_2,c_3$ depending only on the dimension $n$ of $M$ such that
$$
\lambda_k (\Delta_A)\le c_1(\Omega,A)+c_2a^2+c_3\left(\frac{k}{\vert \Omega\vert}\right)^{2/n},
$$
with
$$
c_1(\Omega,A)=\frac{c_1}{\vert \Omega\vert}\left(\frac{\Vert B\Vert^2}
{\lambda_{1,1}(\Omega)}+ d(h,{\cal L}_{\bf Z})^2\right).
$$
\end{thm}

We now collect a few remarks.

\begin{enumerate}

\item In the bound, the only term depending on the potential $A$ is $c_1(\Omega,A)$. The other terms depend only on the geometry of $\Omega$, in particular we observe that when $A=0$ then
$\Delta_A$ is the usual Laplacian $\Delta$ and the bound reduces to :
$$
\lambda_k (\Delta)\le  c_2a^2+c_3\left(\frac{k}{\vert \Omega\vert}\right)^{2/n}.
$$
which was already proved in [CM].

\item
Let us now fix the  Riemannian metric $g$ and consider all possible potentials $A$ on $\Omega$. Then, the basis $A_1,..,A_m$ can be fixed once and for all,  and does not depend on $A$. The dependance on $A$ is only through the distance to the integral lattice ${\cal L}_{\bf Z}$  and 
through $\Vert B\Vert^2$  (note that $B=dA$ is also independent on $g$).  
 
\item
It is known that for a fixed Riemannian metric, the first eigenvalue $\lambda_1(\Delta_A)$ may be arbitrarily large if $\Vert B\Vert$ is large. An example for the round sphere is given in [BCC]. However, it is easy to show that large $\Vert B\Vert$ does not necessarily implies that $\lambda_1(\Delta_A)$ is large.

\item
The bound in terms of $\Vert B\Vert^2$ is probably not optimal, at least in the case of surfaces. In [Er2], Erd\"{o}s has established a bound in terms of $\Vert B\Vert$: the proof is much more complicated than that for the bound involving  $\Vert B\Vert^2$.

\item
We observe that the above  estimates are compatible with the Weyl law. 
\end{enumerate}

The strategy of the proof is to adapt a method developped in [CM]: we construct a suitable family of disjoint subsets of $\Omega$, and a family of test functions for the standard min-max characterization of the spectrum which are associated to these subsets. In the case of the Laplacian acting on functions, we choose as test functions some plateau functions supported on these subsets. However, in our case, the constant functions are no longer parallel for the modified connection $\nabla^A$,  because of the presence of the potential $A$. Therefore, an argument ad-hoc will be applied.

%%%%%%%%

\subsection{Lower bounds}\label{lowerbounds}

We will give lower bounds  for $\lambda_1(\Delta_A)$ when the potential $A$ is a {\it closed} $1$-form and the manifold is a Riemannian cylinder.  These lower bounds are given in terms of the flux of the $1$-form across the generator of the $1$-homology (more precisely, in terms of the distance defined in \eqref{dflux}) and in terms of the geometry of the manifold.  They have the advantage to be sharp and we will actually characterize the equality case, which will force the cylinder to be a Riemannian product (in particular, a flat metric). The disavantage is perhaps that the dependance on the geometry is not very explicit. This is the reason why we will study some specific situations like a revolution cylinder or an annular region in the plane bounded by two convex curves. In this last situation, we will get a lower bound for 
$\lambda_1(\Delta_A)$ where the dependance on the geometry is very explicit.

\medskip
\noindent
\textbf{Cylinders.} By a {\it Riemannian cylinder} we mean a  domain $(\Omega,g)$ diffeomorphic to
$[0,1]\times\sphere 1$ and endowed with a Riemannian metric $g$. 
Note that $\partial \Omega$ has two components: $\partial\Omega=\Sigma_1\cup\Sigma_2$ where
$$
\Sigma_1=\{0\}\times\sphere 1, \quad \Sigma_2=\{1\}\times\sphere 1.
$$
We will need to foliate the cylinder by the (regular) level curves of a smooth function $\psi$. We introduce

\nero ${\cal F}_{\Omega}$, the family of smooth real-valued functions on $\Omega$ which have no critical points in $\Omega$ and which are  constant on each component of the boundary of $\Omega$. 

\medskip

As 
$\Omega$ is a cylinder, we see that ${\cal F}_{\Omega}$ is not empty. 
If $\psi\in{\cal F}_{\Omega}$, we set:
$$
K=K_{\Omega,\psi}=\dfrac{\sup_{\Omega}\abs{\nabla\psi}}{\inf_{\Omega}\abs{\nabla\psi}}.
$$
It is clear that, in the definition of the constant $K$, we can assume that the range of $\psi$ is the interval $[0,1]$, and that $\psi=0$ on $\Sigma_1$ and $\psi=1$ on $\Sigma_2$. 
Note that  the level curves of the function $\psi$ are all smooth, closed and connected; moreover they are all homotopic to each other so that the flux of any closed $1$-form $A$ across any of them is the same, and will be denoted by $\Phi^A$.

\medskip
\noindent
We say, briefly, that $\Omega$ is {\it $K$-foliated by the level curves of $\psi$.} 

\medskip

Finally, we say that $\Omega$ is a Riemannian product if it is isometric to $[0,a]\times\sphere 1(R)$ for suitable positive constants  $a,R$.

\smallskip

We now give a lower bound when $\Omega$ is an arbitrary cylinder. This lower bound will involve the constant $K$ defined above: is is not always easy to estimate $K$. In Section \ref{estimate K} we will show how to estimate $K$ in more general cases,   in terms of the metric tensor. Note that $K\geq 1$; we will see that in many interesting situations (for example, for revolution cylinders) one has in fact $K=1$. 

\begin{thm} \label{main3} \item a) Let $(\Omega,g)$ be a Riemannian cylinder, and let $A$ be a closed $1$-form on $\Omega$. Assume that $\Omega$ is $K$-foliated by the level curves of the smooth function $\psi\in{\cal F}_{\Omega}$. Then:
$$
\lambda_1(\Omega,A)\geq\dfrac{4\pi^2}{KL^2}\cdot d(\Phi^A,{\bf Z})^2,
$$
where $L$ is the maximum length of a level curve of $\psi$ and $\Phi^A$ is the flux of $A$ across any of the boundary components of $\Omega$.

\item b) Equality holds if and only if the cylinder $\Omega$  is  a Riemannian product.
\end{thm}

\nero It is clear that we can also state the lower bound as follows:
$$
\lambda_1(\Omega,A)\geq\dfrac{4\pi^2}{\tilde K_{\Omega}}\cdot d(\Phi^A,{\bf Z})^2,
$$
where $\tilde K_{\Omega}$ is the invariant :
$$
\tilde K_{\Omega}=\inf_{\psi\in{\cal F}_{\Omega}}K_{\Omega,\psi}L_{\psi}^2\quad\text{and}\quad L_{\psi}=\sup_{r\in {\rm range}(\psi)}\abs{\psi^{-1}(r)}.
$$
We will  obtain a good estimate of $K$ for an annular region in the plane. This will give rise to the following  consequence.

\begin{thm} \label{main2} Let $\Omega$ be an annulus in $\real 2$ bounded by the inner curve $\Sigma_1$ and the outer curve $\Sigma_2$, both convex.  Assume that $A$ is a closed potential having flux $\Phi^A$ around $\Sigma_1$. Then:
$$
\lambda_1(\Omega,A)\geq \dfrac{4\pi^2\beta^2}{B^2L^2} d(\Phi^A,{\bf Z})^2
$$
where $\beta$ (resp. $B$) is the minimum (resp. maximum) distance of $\Sigma_1$ to $\Sigma_2$ and $L$ is the length of the outer boundary. 
\end{thm} 

In section \ref{convex}, we will explain why we need to control $\beta$, $B$, $L$, and why we need to impose convexity.
%%%%%%

\section{Upper bounds} \label{upper} In this section, we prove Proposition \ref{harmonic potential} and Theorem \ref{main1}.

\subsection{Proof of Proposition \ref{harmonic potential}}
We  first  show the inequality
$$
\l_1(\Omega,A)\leq\frac{ d(h,{\cal L}_{\bf Z})^2}{\vert \Omega\vert}+\frac{\Vert B\Vert^2}{\lambda_{1,1}(\Omega)\vert \Omega\vert}.
$$
We recall that $A=\delta\psi+h$ denotes the potential, $\psi$ is a smooth $2$-form, $h\in{\rm Har}_1(\Omega)$,  $\lambda_{1,1}(\Omega)$ denotes the first eigenvalue of the Laplacian acting on co-exact $1$-forms,  $B=dA$ is the curvature of the potential $A$, and ${\cal L}_{\bf Z}$ denotes the integral lattice generated by a basis $A_1,\dots,A_m$.

\medskip
\noindent
\textbf{Step 1.} We consider the case where the potential belongs to the integral lattice ${\cal L}_{\bf Z}$. Let us denote it by $\omega$ (that is, $h=\omega$) and write

$$
\omega=n_1A_1+\dots+n_mA_m,
$$
for integers $n_1,\dots,n_m$. We want to verify that the magnetic Laplacian with potential $\omega$ has first eigenvalue equal to zero: we will use this fact in order to construct the test functions in the subsequent steps. Fix a base point $x_0$ and define, for $x\in\Omega$:
\begin{equation}\label{path}
\phi(x)\doteq\int_{x_0}^x\omega,
\end{equation}
where on the right we mean integration of $\omega$ along any path joining $x_0$ with $x$. As $\omega$ is closed, $\phi(x)$ does not depend on the choice of two homotopic paths and since the flux of $A$ across each $c_j$ is an integer, $\phi(x)$ is multivalued and defined up to $2\pi \mathbb Z$. This implies that the function $u(x)=e^{i\phi(x)}$ is well defined. 
As 
$
d\phi=\omega
$
we see that
$
du=iu\omega.
$
Therefore
$$
d^{\omega}u=du-iu\omega=0,
$$
the function $u$ is $\nabla^{\omega}$-parallel and is thus an eigenfunction of $\Delta_{\omega}$ associated 
to the eigenvalue $0$.

\medskip
\noindent
\textbf{Step 2.} We now look at the case where $A=\delta\psi+h$ and $h$ does not necessarily belong to the integral lattice ${\cal L}_{\bf Z}$. Let $\omega$ be an arbitrary form in ${\cal L}_{\bf Z}$, and construct, as above, the function 
$$
u(x)=e^{i\phi(x)},
$$
with $\phi(x)$ as in \eqref{path}. 
Then $du=iu\omega$ and therefore
$$
d^Au=du-iuh-iu\delta \psi=iu(\omega-h-\delta \psi).
$$
Since $\abs{u}=1$, we obtain:
$$
\abs{d^Au}^2=\abs{\omega-h-\delta \psi}^2.
$$
We use $u(x)$ as test-function for the first eigenvalue of $\Delta_A$. Then:
\begin{equation}\label{ltwo}
\lambda_1(\Delta_A)\leq \dfrac{\int_{\Omega}\abs{d^Au}^2}{\int_{\Omega}\abs{u}^2}=\dfrac{\norm{\omega-h-\delta \psi}^2}{\vert\Omega\vert}= \dfrac{\norm{\omega-h}^2+\norm{\delta \psi}^2}{\vert\Omega\vert}.
\end{equation}
The last step follows from the fact that, as $\omega-h$ is harmonic, it is $L^2$-orthogonal to $\delta \psi$.
Now observe that, since $\delta\psi$ is coexact and tangential, one has by the variational characterization of the eigenvalue $\lambda_{1,1}(\Omega,g)$:
$$
\frac{\int_{\Omega} \vert d \delta \psi\vert^2}{\int_{\Omega} \vert \delta \psi\vert^2} \ge \lambda_{1,1}(\Omega,g)
$$
 As $d \delta\psi=B$, we have
\begin{equation}\label{lambdaoneone}
\int_{\Omega} \vert \delta \psi \vert^2 \le \frac{1}{\lambda_{1,1}(\Omega,g)}\Vert B\Vert^2.
\end{equation}
Taking the infimum on the right-hand side of \eqref{ltwo} over all $\omega\in{\cal L}_{\bf Z}$ we obtain, taking into account \eqref{lambdaoneone}:
$$
\l_1(\Delta_A)\leq\frac{ d(h,{\cal L}_{\bf Z})^2}{\vert \Omega\vert}+\frac{\Vert B\Vert^2}{\lambda_{1,1}(\Omega)\vert \Omega\vert}
$$
as asserted.

\subsection{The sharpness} \label{torus} Recall that if $B=0$, we get the estimate 
$$
\l_1(\Delta_A)\leq\frac{ d(A,{\cal L}_{\bf Z})^2}{\vert \Omega\vert}.
$$
Let us show that it is sharp. We have to construct an example where
$
\l_1(\Delta_A)=\frac{ d(A,{\cal L}_{\bf Z})^2}{\vert \Omega\vert}.  
$
This will be the case for rectangular flat tori. We start from the case where $\Omega$ is a circle of radius $R$  and perimeter $L=2\pi R$ and let $A$ be any $1$-form on $\Omega$.  Let $\Phi^A$ be the flux of $A$ across the circle.  By the calculation in Proposition \ref{circle}
we know that the first eigenvalue is
$$
\lambda_1(\sphere 1(R),A)=\dfrac{4\pi^2}{L^2}d(\Phi^A,{\bf Z})^2
$$
where $d(\Phi^A,{\bf Z})$ denotes the minimum distance of $\Phi^A$ to the lattice of integers. 

We first assume that the potential $A$ is harmonic and that the (unique up to isometries) metric is the standard one, written $g=dt^2$ where $t\in [0,L]$ is arc-length.
Now take $c_1(t)=t$ be the cycle which goes around the circle once : its dual $1$-form is $A_1=\frac{2\pi}{L}dt$ and $\omega\in{\cal L}_{\bf Z}$ if and only if
$$
\omega=\dfrac{2\pi k}{L}\,dt
$$
for some integer $k$. Any harmonic $1$-form $A$ is written $A=\Phi^AA_1=\frac{2\pi\Phi^A}{L}\,dt$ and then
$$
\int_{\sphere 1(R)}\abs{A-\omega}^2=L\cdot \dfrac{4\pi^2}{L^2}(\Phi^A-k)^2,
$$
which gives, as $L=\abs{\Omega}$:
$$
\dfrac{1}{\abs{\Omega}}\norm{A-\omega}^2=\dfrac{4\pi^2}{L^2}(\Phi^A-k)^2.
$$
Taking the infimum over all $k\in\bf Z$ and taking into account the above proposition we see
$$
\dfrac{1}{\abs{\Omega}}d(A,{\cal L}_{\bf Z})^2=\lambda_1(\sphere 1(R))
$$
as asserted. If $A$ is any $1$-form, then it must be closed by dimensional reasons, and  by Hodge theorem we can always decompose it
$$
A=h+df
$$
where $h=a\,dt$ is harmonic. Then, by gauge invariance, the spectrum for the potential  $A$ is equal to the spectrum for the potential  $h$; as $\Phi^A=\Phi^h$, and since the spectrum depends only on the flux, the conclusion follows. 

\medskip

Let us now see that we have sharpness for any flat rectangular torus, that is, for any Riemannian product
$$
M=\sphere 1(\frac{p_1}{2\pi})\times\cdots\times \sphere n(\frac{p_n}{2\pi}).
$$
Note that $M$ is  the quotient of $\real n$ by the lattice generated by the basis
$
(p_1e_1,\dots,p_ne_n).
$
By Proposition \ref{torus} in the Appendix we know that the first eigenvalue of the magnetic Laplacian with potential $A$ (a closed $1$-form) is 
$$
\lambda_1(M,A)=\sum_{j=1}^n\dfrac{4\pi^2}{p_j^2}\cdot d(\Phi^A_{j},{\bf Z})^2,
$$
where $\Phi^A_{j}$ is the flux of $A$ across $\sphere 1(\frac{p_j}{2\pi})$. One verifies that the generic harmonic $1$-form in the lattice ${\cal L}_{\bf Z}$ is written $\omega=\sum_{j=1}^n\frac{2\pi k_j}{p_j}\,dt_j$; by Gauge invariance, we can assume that $A$ is harmonic, so that
$$
A=\sum_{j=1}^n\frac{2\pi \Phi^A_{j}}{p_j}dt_j.
$$
Repeating word for word the arguments for the one-dimensional case we arrive at the equality
$$
\lambda_1(M,A)=\dfrac{1}{\abs{\Omega}}d(A,{\cal L}_{\bf Z})^2.
$$
We omit further details.

\begin{rem}
 We observe that, on a fixed rectangular flat torus,  $\lambda_1(M,A)$ is maximal precisely when the flux of $A$ across each factor is congruent to $\frac 12$ modulo integers. 
\end{rem}

\subsection{Proof of Theorem \ref{main1}}

For a general potential $A=h+\delta \psi$, the scope is to prove the following estimate

$$
\lambda_k (\Delta_A)\le c_1(\Omega,A)+c_2a^2+c_3\left(\frac{k}{\vert \Omega \vert }\right)^{2/n},
$$
with
 $$
c_1(\Omega,A)=\frac{c_1}{\vert \Omega \vert}\left(\frac{\Vert B\Vert^2}{\lambda_{1,1}}+ d(h,{\cal L}_{\bf Z})^2\right),
$$
and $c_1,c_2,c_3$ are constants depending only on the dimension $n$ of $M$.

\begin{proof}

Note first that, if the first cohomology $H^1(\Omega)=0$, we have $h=0$ and $d(h,{\cal L}_{\bf Z})=0$.

\medskip
We suppose now that $A$ has the general form $A=h+ \delta \psi$ with $dh=\delta h=0$.
Let $(A_1,\dots,A_m)$ be a basis in ${\rm Har}_1(\Omega)$ dual to a given homology basis, let
${\cal L}_{\bf Z}$ be the integral lattice generated by $A_1,\dots,A_m$ and fix $\omega\in {\cal L}_{\bf Z}$ at minimum distance to $h$. This means that
$$
\norm{\omega-h}^2=d(h,{\cal L}_{\bf Z})^2.
$$
In order to construct a family of test functions for the Rayleigh quotient 
$$
R(s) = \frac{\int_{\Omega} \vert (d-iA)s \vert^2 }{\int_{\Omega} \vert s\vert^2}, 
$$
we use the function $u$, eigenfunction for the eigenvalue $0$ of the Laplacian associated to the potential
$\omega$. Then (see the proof of Proposition \ref{harmonic potential}):
$$
du=iu\omega.
$$
We consider functions of the type $s(x)=f(x)u(x)$, where $f$ is a real function.
We have
$$
\begin{aligned}
(d-iA)(fu)&= u df+ fdu-ihuf-iuf\delta \psi \\
&=udf+iuf(\omega-h-\delta \psi) .
\end{aligned}
$$
Since $\vert u\vert=1$:
$$
\vert (d-iA)(fu)\vert^2 \le 2\Big(\vert df \vert^2 +f^2\vert \omega-h-\delta \psi\vert^2  \Big).
$$
We have to control the Rayleigh quotient
$$
R(fu) \le 2\left(\frac{\int_{\Omega} \vert df\vert^2}{\int_{\Omega} f^2}  + \frac{\int_{\Omega} f^2\vert \omega-h-\delta \psi\vert^2}{\int_{\Omega} f^2} \right).
$$
So, we are lead to control the Rayleigh quotient
$$
R(f)=2 \frac{\int_{\Omega}\vert df\vert^2+Vf^2}{\int_{\Omega} f^2}, \quad\text{where $V= \vert \omega-h-\delta \psi\vert^2$.}
$$

Thus, the problem is now to find an upper bound for the spectrum of the operator $\Delta +V$, where $\Delta$ is the usual Laplacian acting on functions and $V=\vert \omega-h-\delta \psi\vert^2 $ is a nonnegative potential. 

\medskip
The idea is to construct a family of disjointed supported domains, and to construct test-functions supported in these domains.

\medskip
First, we show how to control the $L^1$-norm of the potential $V$. We are exactly as in Step 2 of the proof of Proposition \ref{harmonic potential} and we get
 
\begin{equation}\label{potential}
\int_{\Omega}V\leq   \frac{1}{\lambda_{1,1}(\Omega,g)}\Vert B\Vert^2+d(h,{\cal L}_{\bf Z})^2.
\end{equation}

\medskip
\noindent
\textbf{Control of 
$$
R(f)=\frac{\int_{\Omega}\vert df\vert^2+Vf^2}{\int_{\Omega} f^2}.
$$} 
This can be deduced from [Ha] as consequence of a rather difficult result. However, for convenience of the reader, we will give a simpler proof following the argument in [CM], where, in the same context (a domain $\Omega \subset M$, complete manifold with Ricci curvature $\ge -(n-1)a^2$) we have constructed upper bounds for the Laplacian without potential. It turns out that we can follow word for word the method of [CM] with some slight modifications, using  corollary 2.3 and  lemma 3.1 of [CM]. We assume, without loss of generality, that $a=1$.

\medskip

First, we recall the definition of {\it packing constant} : for $r>0$, we define 
$$
C(r)=\text{number of balls of radius $r$ needed to cover a ball of radius $4r$.}
$$
As the volume of $r$-balls tends uniformly to $0$ as $r\to 0$, we see that there exists $r_0>0$ such that, for all $r\leq r_0$:
\begin{equation}\label{techniq}
	{2C(r)\abs{B(x,r)} \le \alpha:= \frac{\abs{\Omega}}{6C(r)k}} \ ,
\end{equation}
holds for all $x\in M$. By Corollary 2.3 in [CM], we get  the existence of $3k$ measurable subsets $\Omega_1,...,\Omega_{3k}$ each of measure 
$\abs{\Omega_i} \ge \frac{\abs{\Omega}}{6C(r)k}$ and such that  $d(\Omega_i,\Omega_j)\ge 3r$ if $i\not =j$, where $d$ denotes the usual Riemannian distance. For a subset $A\in\Omega$ we denote by $A^r$ the tubular neighborhood of $A$ in $\Omega$, with radius $r$ :
$$
A^r=\{x\in\Omega: d(x,A)<r\}.
$$
Then one sees that  $\Omega_i^r$ and $\Omega_j^r$, are also pairwise disjoint.

We can now apply the construction of Lemma 3.1 in [CM] and get an $H^1(\Omega)$-orthogonal 
family of $3k$ test functions $\left(f_{j}\right)^{3k}_{j=1}$ of disjoint support  (each $f_j$ being supported in $\Omega_j^r$) whose Rayleigh quotient satisfies
\begin{equation}\label{rq}
R(f_i) \le \frac{1}{r^2}\frac{\abs{\Omega_i^r\setminus\Omega_i}}{\abs{\Omega_i}}+\frac{\int_{\Omega}Vf_i^2}{\int_{\Omega}f_i^2}
\end{equation}
(the second term on the right hand side does not appear in [CM] because there the operator is the usual Laplacian, while here is $\Delta+V$). Let 
$$
Q= \sharp \left\{i\in \left\{1,...,3k \right\}: \abs{\Omega_i^r} \ge \frac{\abs{\Omega}}{k}\right\}.
$$
Clearly $Q \le k$, so that for at least $2k$ of these $3k$ subsets $\Omega_1,...,\Omega_{3k}$ we have $\abs{\Omega_i^r}\le \frac{\abs{\Omega}}{k}$. 
In turn, for at least $k$ of these $2k$ subset, we have
$$
\int_{\Omega_i^r} V\leq \frac{1}{k}\int_{\Omega} V.
$$
The conclusion is that, for at least $k$ of the above subsets, which, after renumbering, we can assume to be $\Omega_1,\dots,\Omega_k,$ we have at the same time :
\begin{equation}\label{ir}
\abs{\Omega_i^r}\le \frac{\abs{\Omega}}{k}\quad\text{and}\quad \ \int_{\Omega_i^r} V\le \frac{1}{k}\int_{\Omega} V.
\end{equation}

For every such $\Omega_i$, we construct  the corresponding plateau function $f_i$ as Lemma 3.1 of [CM], and use $f_i$ as test function; as $\abs{\Omega_i^r\setminus \Omega_i}\le \frac{\abs{\Omega}}{k}$ and $\abs{\Omega_i}\ge \alpha=\frac{\abs{\Omega}}{6C(r)k}$, taking also into account \eqref{ir},  we arrive at the inequality
\begin{equation}\label{bornepreuve}
R(f_i) \le \frac{1}{r^2} \frac{\abs{\Omega}/k}{\abs{\Omega}/(6C(r)k)}+\frac{\int_{\Omega}Vf_i^2}{\int_{\Omega}f_i^2}\le\frac{6C(r)}{r^2} +6C(r)\frac{\int_{\Omega}V}{\abs{\Omega}}\ .
\end{equation}
 
Let $\omega'_n>0$  be a positive
constant such that $\abs{B(x,r)}\le \omega'_n r^{n}$ for all $r\le1$ in the
hyperbolic space of curvature $-1$. We define the integer $k_0 = \left[\frac{\abs{\Omega}}{12C(1)^2 \omega'_n
}\right] +1 $ and, for every $k\ge k_0$, we set
$$ 
r_k = \left(\frac{\abs{\Omega}}{k}  \frac{1}{12 C(1)^{2} \omega'_n
}\right)^{1/n} \ .
$$
By its definition, $r_k\le 1$ and (\ref{techniq}) holds, since 
$$
{12C(r_k)^2\abs{B(x,r_k)} \le 12C(1)^{2}\omega'_n r_k^n = \frac{\abs{\Omega}}{k}}.
$$
Our inequality \eqref{bornepreuve} becomes: for all $k \ge k_0$,
\begin{equation} \label{ineq}
 \l_k \le \frac{6C(1)}{r_k^2}+6C(1)\frac{\int_{\Omega}V}{\abs{\Omega}}=
6C(1)\Big(12 C(1)^{2} \omega'_n
\Big)^{2/n}\left(\frac{k}{\abs{\Omega}}\right)^{2/n}+6C(1)\frac{\int_{\Omega}V}{\abs{\Omega}}\
\end{equation}
On the other hand, if $k < k_0$, then we obviously have $\l_k \le \l_{k_0}$. Taking into account \eqref{ineq}  we obtain, for all positive integers $k$:
\begin{equation}\label{asympt}
\l_k \le \l_{k_0}  + B_n
\left(\frac{k}{\abs{\Omega}}\right)^{2/n}+6C(1)\frac{\int_{\Omega}V}{\abs{\Omega}} \ ,
\end{equation}
where we have set $B_n := 6C(1)\Big(12 C(1)^{2}\omega'_n \Big)^{2/n}$. 
We now estimate the eigenvalue $\lambda_{k_0}$.
If $k_0=1$ then, taking the constant function $1$ as test-function for the first eigenvalue of $\Delta+V$, one sees that 
\begin{equation}\label{lambdaone}
\lambda_{k_0}\leq \dfrac{1}{\abs{\Omega}}\int_{\Omega}V.
\end{equation}
If $k_0\geq 2$ then  $\frac{\abs{\Omega}}{12C(1)^2\omega'_n }\ge 1$ which implies that, from the definition of $k_0$, we have $k_0\le 2 \frac{\abs{\Omega}}{12C(1)^2\omega'_n }$. Then :
$$
r_{k_0}^2=\left(\frac{\abs{\Omega}}{k_0}  \frac{1}{12 C(1)^{2} \omega'_n}\right)^{2/n}\geq \dfrac{1}{2^{2/n}}.
$$
We now apply \eqref{ineq}  to $k=k_0$ and obtain:
\begin{equation}\label{third}
\l_{k_0} \le
\frac{6C(1)}{r_{k_0}^2}+6C(1)\frac{\int_{\Omega}V}{\abs{\Omega}}=
6C(1) 2^{2/n}+6C(1)\frac{\int_{\Omega}V}{\abs{\Omega}}\ ,
\end{equation}
Comparison of \eqref{asympt}, \eqref{lambdaone} and \eqref{third} says that, for all positive integers $k$, we have:
$$
\lambda_k\leq \alpha_n\frac{\int_{\Omega}V}{\abs{\Omega}}+\beta_n+ B_n
\left(\frac{k}{\abs{\Omega}}\right)^{2/n}
$$
for positive constants $\alpha_n,\beta_n$. To finish the proof, we just have to recall that, from \eqref{potential}:
$$
\frac{\int_{\Omega}V}{\abs{\Omega}}\leq \frac{1}{\abs{\Omega}}\Big(\frac{\abs{B}^2}{\lambda_{1,1}}+d(h,{\cal L}_{\bf Z})^2\Big).
$$
\end{proof}

%%%%%
\subsection{Another upper bound}  In this short paragraph, we give a simple way to get an upper bound when the potential $A$ is \emph{closed}. This will be used in Example \ref{example1} below. The geometric idea is the following: if we have a region $D \subset \Omega$ such that the first absolute cohomology group $H^1(D)$ is $0$, then we can estimate from above the spectrum of $\Delta_A$ in $\Omega$ in terms of the spectrum of the usual Laplacian on $D$. 
The reason is that the potential $A$ is $0$ on $D$ up to a Gauge transformation; then,  on $D$, $\Delta_A$ becomes the usual Laplacian  and any eigenfunction of the Laplacian on $D$ may be extended by $0$ on $\Omega$ and thus used as a test function for the magnetic Laplacian on the whole of $\Omega$. 

\smallskip

Let us give the details.  Let $D$ be a closed subset of $\Omega$ such that, for some (small) $\delta>0$ one has $H^1(D^{\delta},\reals)=0$, where  $D^{\delta}=\{p\in \Omega: {\rm dist}(p,D) < \delta\}$. This happens when $D^{\delta}$ has a retraction onto $D$. We write
$$
\partial D= (\partial D\cap \partial \Omega) \cup (\partial D \cap \Omega)=\partial^{\rm ext}D\cup\partial^{\rm int}D
$$
and we denote by $(\nu_j(D))_{j=1}^{\infty}$ the spectrum of the Laplacian acting on functions, with the Neumann boundary condition on $\partial^{\rm ext}D$ (if non empty) and the Dirichlet boundary condition on $\partial^{\rm int}D$.

\begin{prop}  \label{upperharmonic} Let $\Omega$ be a compact domain of a complete manifold $M$ with smooth boundary $\bd\Omega$. Let $A$ be a closed potential on $\Omega$, and let  $D\subset \Omega$ be a compact subdomain such that $H^1(D,\textbf R)=H^1(D^{\delta},\textbf R)=0$ for some $\delta>0$. Then we have
$$
\lambda_k(\Omega,A) \le \nu_k(D)
$$
for each $k\geq 1$.
\end{prop}

\noindent
\textbf{Proof.} We recall that for any function $\phi$ on $\Omega$, the operator $\Delta_A$ and $\Delta_{A+d\phi}$ are unitarily equivalent and have the same spectrum. As $A$ is closed and, by assumption,  $H^1(D^{\delta},\textbf R)=0$, $A$ is exact on $D^{\delta}$ and there exists a function $\tilde \phi$ on $D^{\delta}$ such that $A+d\tilde\phi=0$ on $D^{\delta}$. 

\smallskip
We consider the restriction of $\tilde\phi$ to $D$ and extend it differentiably on $\Omega$ by using a partition of unity $(\chi_1,\chi_2)$ subordinated to $(D^{\delta},\Omega/D)$. Then, setting
$$
\phi\doteq\chi_1 \tilde\phi
$$ 
we see that $\phi$ is a smooth function on $\Omega$ which is equal to $\tilde\phi$ on $D$ so that, on $D$, one has $A+d\phi=0$.
We consider the new potential $\tilde A=A+d\phi$ and observe that $\tilde A=0$ on $D$.

\smallskip

Now consider an eigenfunction $f$ for the mixed problem on $D$ (Neumann boundary conditions on $\partial^{\rm ext}D$ and  Dirichlet boundary conditions on $\partial^{\rm int}D$), and  extend it by $0$ on $\Omega\setminus D$. As $\tilde A=0$ on $D$, we see that
$$
\abs{\nabla^{\tilde A}f}^2=\abs{\nabla f}^2,
$$ 
and we get a test function  having the same Rayleigh quotient as that of $f$. Thanks to the usual min-max characterization of the spectrum, we obtain
$$
\lambda_k(\Omega,A) = \lambda_k(\Omega,\tilde A)\le \nu_k(D)
$$
for all $k$.
%%%%%
%%%

\section{A lower bound for cylinders} \label{cylinder}

Recall that the scope is to show the estimate
\begin{equation}\label{estimate}
\lambda_1(\Omega,A)\geq\dfrac{4\pi^2}{KL^2}\cdot d(\Phi^A,{\bf Z})^2,
\end{equation}
for a Riemannian cylinder $\Omega$, and to show that equality holds if and only if $\Omega$ is a Riemannian product. The potential is a closed $1$-form $A$ and we assume
that $\Omega$ is $K$-foliated by the level curves of the smooth function $\psi\in{\cal F}_{\Omega}$. The most involved part of the proof is to show  that, in case of equality, the cylinder is a Riemannian product.

%%%

\subsection{Proof of the lower bound}  Fix a first eigenfunction $u$ associated to $\lambda_1(\Omega, A)$ and fix a level curve
$$
\Sigma_r=\{\psi=r\}, \quad\text{where $r\in [0,1]$.}
$$
As $\psi$ has no critical points, $\Sigma_r$ is isometric to $\sphere 1(\frac{L_r}{2\pi})$, where $L_r$ is the length of $\Sigma_r$.  The restriction of $A$ to $\Sigma_r$ is a closed $1$-form denoted by $\tilde A$; we use the restriction of $u$ to $\Sigma_r$ as a test-function for the first eigenvalue $\lambda_1(\Sigma_r,\tilde A)$ and obtain:
\begin{equation}\label{level}
\lambda_1(\Sigma_r,\tilde A)\int_{\Sigma_r}\abs{u}^2\leq\int_{\Sigma_r}\abs{\nabla^{\tilde A}u}^2.
\end{equation}
By the previous estimate on circles:
$$
\lambda_1(\Sigma_r,\tilde A)=\dfrac{4\pi^2}{L_r^2}d(\Phi^{\tilde A},{\bf Z})^2,
$$
where $\Phi^{\tilde A}$ is the flux of $\tilde A$ across $\Sigma_r$. Now note that $\Phi^{\tilde A}=\Phi^{A}$, because $\tilde A$ is the restriction of $A$ to $\Sigma_r$; moreover $L_r\leq L$ 
by the definition of $L$. Therefore:
\begin{equation}\label{llower}
\lambda_1(\Sigma_r,\tilde A)\geq \dfrac{4\pi^2}{L^2}d(\Phi^{ A},{\bf Z})^2
\end{equation}
for all $r$. Let $X$ be a unit vector tangent to $\Sigma_r$. Then:
$$
\begin{aligned}
\nabla^{\tilde A}_{X}u&=\nabla_{X}u-i\tilde A(X)u\\
&=\nabla_{X}u-iA(X)u\\
&=\nabla^A_{X}u.
\end{aligned}
$$
The consequence is that:
\begin{equation}\label{energy}
\abs{\nabla^{\tilde A}u}^2=\abs{\nabla^{\tilde A}_{X}u}^2=\abs{\nabla^{A}_{X}u}^2\leq \abs{\nabla^{A}u}^2.
\end{equation}
\nero {\it Note that equality holds in \eqref{energy} iff $\nabla^A_{N}u=0$ where $N$ is a unit vector normal to the level curve $\Sigma_r$ (we could take $N=\nabla\psi/\abs{\nabla\psi}$).}

\smallskip

For any fixed level curve $\Sigma_r=\{\psi=r\}$ we then have, taking into account \eqref{level}, \eqref{llower} and \eqref{energy}:
$$
\dfrac{4\pi^2}{L^2}d(\Phi^{ A},{\bf Z})^2\int_{\psi=r}\abs{u}^2\leq \int_{\psi=r}\abs{\nabla^Au}^2.
$$
Assume that $B_1\leq\abs{\nabla\psi}\leq B_2$ for positive constants $B_1,B_2$. Then the above inequality implies:
$$
\dfrac{4\pi^2}{L^2}d(\Phi^{ A},{\bf Z})^2\cdot B_1\int_{\psi=r}\dfrac{\abs{u}^2}{\abs{\nabla\psi}}\leq B_2\int_{\psi=r}\dfrac{\abs{\nabla^Au}^2}{\abs{\nabla\psi}}.
$$
We now integrate both sides from $r=0$ to $r=1$ and use the coarea formula. Conclude that
$$
\dfrac{4\pi^2}{L^2}d(\Phi^{ A},{\bf Z})^2\cdot B_1\int_{\Omega}{\abs{u}^2}\leq B_2\int_{\Omega}\abs{\nabla^Au}^2.
$$
As $u$ is a first eigenfunction, one has:
$$
\int_{\Omega}\abs{\nabla^Au}^2=\lambda_1(\Omega,A)\int_{\Omega}\abs{u}^2.
$$
Recalling that $K=\frac{B_2}{B_1}$ we finally obtain the estimate \eqref{estimate}.  

%%%%%%

\subsection{Proof of the equality case}

It is clear that if $\Omega$ is a Riemannian product then we have equality. Now assume that we do have equality: we have to show that $\Omega$ is a Riemannian product.  Going back to the proof, we must have the following facts.

\medskip

{\bf F1.}  {\it All level curves of $\psi$ have the same length $L$}.

\medskip

{\bf F2.}  {\it $\abs{\nabla\psi}$ must be constant and, by renormalization, we can assume that it is everywhere equal to $1$. }Then, $\psi:\Omega\to [0,a]$ for some $a>0$ and we set
$$
N\doteq\nabla\psi.
$$

\medskip

{\bf F3.} {\it The eigenfunction $u$ on $\Omega$ restricts to an eigenfunction of the magnetic Laplacian of each level set $\Sigma_r=\{\psi=r\}$, with potential given by the restriction of $A$ to $\Sigma_r$.}

\medskip

{\bf F4.} {\it One has $\nabla^A_Nu=0$ identically on $\Omega$. }

\subsubsection{First step: description of the metric}

\begin{lemme} 
$\Omega$ is isometric to the product $[0,a]\times \sphere 1(\frac{L}{2\pi})$ with metric
$$
g=\twomatrix 100{\theta^2(r,t)}
$$
where $\theta(r,t)$ is positive and periodic of period $L$ in the variable $t$.
\end{lemme}

We first show that the integral curves of $N$ are geodesics; for this it is enough to show that 
$
\nabla_NN=0
$ 
on $\Omega$. Let $e_1(x)$ be a vector tangent to the level curve of $\psi$ passing through $x$. Then, we obtain a smooth vector field $e_1$ which, together with $N$, forms a global orthonormal frame.  Now 
$$
\scal{\nabla_NN}{N}=\dfrac 12 N\cdot\scal{N}{N}=0.
$$
On the other hand, as the Hessian is a symmetric tensor:
$$
\scal{\nabla_NN}{e_1}=\nabla^2\psi(N,e_1)=\nabla^2\psi(e_1,N)=\scal{\nabla_{e_1}N}{N}=\dfrac 12e_1\cdot\scal{N}{N}=0.
$$
Hence $\nabla_NN=0$ as asserted. As each integral curve of $N=\nabla\psi$ is a geodesic meeting $\Sigma_1$ orthogonally, we see that $\psi$ is actually the distance function to $\Sigma_1$. We introduce coordinates on $\Omega$ as follows. For a fixed point $p\in\Omega$ consider the unique integral curve $\gamma$ of $N$ passing through $p$ and   let $x\in\Sigma_1$ be the intersection 
of $\gamma$ with $\Sigma_1$ (note that $x$ is the foot of the unique geodesic which minimizes the distance from $p$ to $\Sigma_1$). Let $r$ be the distance of $p$ to $\Sigma_1$. We then have a map
$
\Omega\to [0,a]\times\Sigma_1
$
which sends $p$ to $(r,x)$. Its  inverse is the map $F: [0,a]\times\Sigma_1\to\Omega$ defined by
$$
F(r,x)=\exp_x(rN).
$$
Note that $F$ is a diffeomeorphism; we call the pair $(r,x)$ the {\it normal coordinates} based on $\Sigma_1$. We introduce the arc-length $t$ on $\Sigma_1$ (with origin in any assigned point of $\Sigma_1$) and call $L$ the length of $\Sigma_1$. Then,  the pull-back metric on $[0,a]\times [0,L]$ takes the form:
$$
g=\twomatrix 100{\theta^2}
$$
for a positive smooth function $\theta=\theta(r,t)$ such that $\theta(r,0)=\theta(r,L)$.  In fact, since $N=\opd r$ one sees that
$g_{11}=1$ everywhere; for any fixed $r=r_0$ we have that $F(r_0,\cdot)$ maps $\Sigma_1$ diffeomorphically onto the level set $\{\psi=r_0\}$ so that $\opd r$ and $\opd t$ will be mapped onto orthogonal vectors, and indeed $g_{12}=0$.
Observe that the pair:
$$
N=\opd r, \quad e_1=\dfrac{1}{\theta} \opd t
$$
is  a global  orthonormal frame. Finally note that
$$
\theta(0,t)=1
$$
for all $t$, because $F(0,\cdot)$ is the identity. 

%%%%%%%

\subsubsection{Second step : Gauge invariance} 

\begin{lemme} Let $A=f(r,t)\,dr+h(r,t)\,dt$ be a closed $1$-form on $\Omega$. Then, there exists a smooth function $\phi$ on $\Omega$ such that
$$
A+d\phi=H(t)\,dt
$$
for a smooth function $H(t)$ depending only on t. Hence, by Gauge invariance, we can actually assume that $A=H(t)\,dt$.
\end{lemme}

\begin{proof} Consider the function
$$
\phi(r,t)=-\int_0^rf(x,t)\,dx.
$$
Then:
$$
A+d\phi=\tilde h(r,t)\,dt
$$
for some smooth function $\tilde h(r,t)$. 
As $A$ is closed, also $A+d\phi$ is closed, which implies that $\derive{\tilde h}{r}=0$, that is,
$
\tilde h(t,r)
$
does not depend on $r$; if we set $H(t)\doteq\tilde h(t,r)$ we get the assertion.
\end{proof}

We point out the following consequence. If $u=u(r,t)$ is an eigenfunction, we know from  {\bf F4} above that $\nabla^A_Nu=0$, where $N=\derive{}{r}$. As 
$$
\nabla^A_Nu=\derive ur-iA(\derive{}{r})u
$$
and $A=H(t)\,dt$ we obtain $A(\derive{}{r})=0$ hence
$
\derive ur=0
$
at all points of $\Omega$ and  
$$
u=u(t)
$$
 depends only on $t$. 

%%%%%%%%%%%

\subsubsection{Third step : the spectrum of a circle}  In this section, we give an expression for the eigenfunctions of the magnetic Laplacian on a circle with a Riemannian metric $g$ and a closed potential $A$. Of course, we know that any metric $g$ on a  circle is always isometric to the canonical metric $g_{\rm can}=\,dt^2$, where $t$ is arc-length. But our problem in this proof is to reconstruct the global metric of the cylinder and to show that it is a product, and we cannot suppose a priori that the restricted metric of each level set of $\psi$  is the canonical metric. The same is true for the restricted potential: we know that it is Gauge equivalent to a potential of the type $a\,dt$ for a scalar $a$, but we cannot suppose a priori that it is of that form.

We refer to Appendix \ref{riemannian circle} for the complete proof of the following fact.
\begin{prop}\label{circle} Let $(M,g)$ be the circle of length $L$ endowed with the metric
$
g=\theta(t)^2\,dt^2
$
where $t\in [0,L]$ and $\theta(t)$ is a positive function, periodic of period $L$. Let $A=H(t)\,dt$. Then, the eigenvalues of the magnetic Laplacian with potential $A$ are:
$$
\lambda_k(M,A)=\dfrac{4\pi^2}{L^2}(k-\Phi^A)^2, \quad k\in\bf Z
$$
with associated eigenfunctions
$$
u_k(t)=e^{i\phi(t)}e^{\frac{2\pi i (k-\Phi^A)}{L}s(t)}, \quad k\in\bf Z.
$$
where $\phi(t)=\int_0^tH(\tau)\,d\tau$ and $s(t)=\int_0^t\theta(\tau)\,d\tau$. 

\smallskip

In particular, if the metric is the canonical one, that is, $g=dt^2$, and the potential $1$-form is harmonic, so that $A=\frac{2\pi \Phi^A}{L}dt$, then the eigenfunctions are
simply :
$$
u_k(t)=e^{\frac{2\pi i k}{L}t}, \quad k\in\bf Z.
$$
\end{prop}

We remark that if the flux $\Phi^A$ is not congruent to $1/2$ modulo integers, then the eigenvalues are all simple. If the flux is congruent to $1/2$ modulo integers, then there are two consecutive integers $k,k+1$ such that
$$
\lambda_{k}=\lambda_{k+1}.
$$
Consequently, the lowest eigenvalue has multiplicity two, and the first eigenspace is spanned by
$$
e^{i\phi(t)}e^{\frac{\pi i}{L}s(t)}, \, e^{i\phi(t)}e^{-\frac{\pi i}{L}s(t)}.
$$

%%%%%

\subsubsection{Fourth step : a calculus lemma} In this section, we state a technical lemma which will allow us to conclude. The proof is conceptually simple, but perhaps tricky at some points; then,  we decided to put it in Appendix \ref{technical lemma}.

\begin{lemme} \label{calculus} Let $s:[0,a]\times [0,L]\to \reals$ be a smooth, non-negative function such that
$$
s(0,t)=t,\quad s(r,0)=0, \quad s(r,L)=L \quad\text{and}\quad \derive st(r,t)\doteq\theta(r,t)>0.
$$
Assume that there exist smooth functions $p(r),q(r)$ with $p(r)^2+q(r)^2$ not identically zero, such that
$$
p(r)\cos(\frac{\pi}{L}s(r,t))+q(r)\sin(\frac{\pi}{L}s(r,t))=F(t)
$$
where $F(t)$ depends only on $t$. Then $p$ and $q$ are constant and
$
\derive sr=0
$
so that 
$$s(r,t)=t
$$
 for all $(r,t)$.
\end{lemme}

%%%%%%

\subsubsection{End of proof of the equality case} Assume that equality holds. Then, if $u$ is an eigenfunction, we know that $u=u(t)$ and $u$ restricts to an eigenfunction on each level circle $\Sigma_r$ for the potential  $A=H(t)\,dt$ above (see {\bf F3}  and the second step above). 

\medskip
We assume that $\Phi^A$ is congruent to $\frac 12$ modulo integers. This is the most difficult case; in the other cases the proof is a particular case of this, it is simpler and we omit it.

%\nero For simplicity of notation, we assume that $a=L=1$. This will not affect generality. 

\medskip

Recall that each level set $\Sigma_r$ is a circle of length $L$ for all $r$, with metric $g=\theta(r,t)^2\,dt$. As the flux of $A$ is congruent to $\frac 12$ modulo integers, we see that there exist complex-valued functions $w_1(r),w_2(r)$ such that
$$
u(t)=e^{i\phi(t)}\Big(w_1(r)e^{\frac{\pi i}{L}  s(r,t)}+w_2(r)e^{-\frac{\pi i}{L}  s(r,t)}\Big),
$$
which, setting $f(t)=e^{-i\phi(t)}u(t)$,  we can re-write
\begin{equation}\label{rewrite}
f(t)=w_1(r)e^{\frac{\pi i}{L}  s(r,t)}+w_2(r)e^{-\frac{\pi i}{L} s(r,t)}.
\end{equation}
Recall that here $\phi(t)=\int_0^tH(\tau)\,d\tau$ and
$$
s(r,t)=\int_0^t\theta(r,\tau)\,d\tau.
$$
We take the real part on both sides of \eqref{rewrite} and obtain smooth real-valued functions $F(t), p(r),q(r)$ such that 
$$
F(t)=p(r)\cos({\frac{\pi}{L}}s(r,t))+q(r)\sin(\frac{\pi}{L} s(r,t)).
$$
Since $\theta(0,t)=1$ for all $t$, we see
$$
s(0,t)=t.
$$
Clearly $s(r,0)=0$; finally, $s(r,L)=\int_0^L\theta(r,\tau)\,d\tau=L$, being the length of the level circle $\Sigma_r$. Thus, we can apply Lemma \ref{calculus} and conclude that $s(r,t)=t$ for all $t$, that is,
$$
\theta(r,t)=1
$$
for all $(r,t)$ and the metric is a Riemannian product. 

It might happen that $p(r)=q(r)\equiv 0$. But then the real part of $f(t)$ is zero and we can work in an analogous way with the imaginary part of $f(t)$, which cannot vanish unless $u\equiv 0$. 

%%%%%%

%%%%%%%%%

\subsection{General estimate of $K_{\Omega,\psi}$} \label{estimate K}

To estimate $K_{\Omega,\psi}$ for a cylinder $\Omega=[0,a]\times\sphere 1$ we can use the explicit expression of the metric in the normal coordinates $(r,t)$, where $t\in [0,2\pi]$ is arc-length :
$$
g=
\left(
    \begin{array}{cc}
	g_{11} &  g_{12}  \\
	 g_{21} & g_{22}
	 \end{array}
		 \right).
$$
If $g^{ij}$ is the inverse matrix of $g_{ij}$, and if $\psi=\psi(r,t)$ one has:
$$
\abs{\nabla\psi}^2=g^{11}\Big(\derive{\psi}{r}\Big)^2+2g^{12}\derive{\psi}r\derive{\psi}{t}+
g^{22}\Big(\derive{\psi}{t}\Big)^2.
$$
The function $\psi(r,t)=r$ belongs to ${\cal F}_{\Omega}$ and one has:
$
\abs{\nabla\psi}^2=g^{11},
$
which immediately implies that we can take
$$
K_{\Omega,\psi}\leq \dfrac{\sup_{\Omega}g^{11}}{\inf_{\Omega}g^{11}}.
$$
Note in particular that if $\Omega$ is rotationally invariant, so that the metric writes as
$$
g=\left(
    \begin{array}{cc}
	1 &  0 \\
	0 & \theta(r)
	 \end{array}
		 \right),
$$
 then $K_{\Omega,\psi}=1$. The estimate becomes
\begin{equation}\label{simple}
\lambda_1(\Omega,A)\geq\dfrac{4\pi^2}{L^2}\cdot d(\Phi^A,{\bf Z})^2,
\end{equation}
where $L$ is the maximum length of a level curve $r={\rm const}$.

\begin{ex}
{\rm Yet more generally, one can fix a smooth closed simple curve $\gamma$ on a Riemann surface $M$ and consider the tube of radius $R$ around $\gamma$:
$$
\Omega=\{x\in M: d(x,\gamma)\leq R\}.
$$
It is well-known that if $R$ is sufficiently small then $\Omega$ is a cylinder with smooth boundary which can be foliated by the level sets of $\psi$, the distance  function to $\gamma$.
Clearly $\abs{\nabla\psi}=1$ and \eqref{simple} holds as well. 

A concrete example where we  could estimate the width $R$ is the case of a compact surface $M$ of genus $\ge 2$ and curvature $-a^2\le K\le -b^2$, $a \ge b >0$. If we consider a simple closed geodesic $\gamma$, it is known that it has a neighborhood of size $C(\gamma,a)$ (depending on $a$ and of the length of $\gamma$) which is diffeomorphic to the product $S^1 \times (-1,1)$ (see for example [CF]). Then we can take as Riemannian cylinder $\Omega$ a cylinder having one boundary component equal to $\gamma$  and of length $L <C(\gamma,a)$. Because of the Gauss-Bonnet Theorem, we can foliate $\Omega$  with the  level sets of the distance function to $\gamma$; then $K=1$ and \eqref{simple} holds  with $L$ given by the length of the other boundary component.}
 
\end{ex}

%%%%

\subsection{Proof of Theorem \ref{main2}: plane annuli} \label{convex}
In this section, we study a situation where we can construct a foliation with a good estimate of the constant $K$ in terms of the geometry. As a consequence, we get a lower bound for $\lambda_1(\Omega,A)$ depending on natural geometric invariants associated to the distance.

The set-up is the following: $\Omega$ is a smooth, doubly connected domain in the plane, bounded by the outer curve $\Sigma_2$ and the inner curve $\Sigma_1$, both convex. Assume that $A$ is a closed potential having flux $\Phi^A$ around $\Sigma_1$. Then, the goal is to prove the estimate
$$
\lambda_1(\Omega,A)\geq \dfrac{4\pi^2\beta^2}{B^2L^2} d(\Phi^A,{\bf Z})^2
$$
where $\beta$ (resp. $B$) is the minimum (resp. maximum) distance of $\Sigma_1$ to $\Sigma_2$ and $L$ is the length of the outer boundary.  

\medskip
Before giving the proof, let us discuss briefly the result. It is clear that, in order to have a lower bound, we need to control the length $L$ and the maximal distance $B$ between the two boundaries : it suffices to think to an annulus in the plane with a small inner circle and a large outer circle: as it contains a big disk, the upper bound we got in Proposition \ref{upperharmonic} implies that the first eigenvalue of the domain (for a closed potential) is bounded from above by the first eigenvalue of the Dirichlet problem of the disk, which is known to tend to $0$ as the radius of the disk tends to $\infty$. 

It is less obvious that it is also necessary to control the minimum distance $\beta$ between the two boundaries, and we will construct a specific example (Example \ref{example1} below) which will show this fact. We will also explain why we need  convexity: with a local and arbitrarily small (but not convex) deformation, we can completely perturb the spectrum. Note, also, that we cannot hope to control the spectrum simply by imposing a bound on the curvature of the boundary $\Sigma_1 \cup \Sigma_2$ (see Example \ref{example2} below).

 \medskip
Now, let us begin with the proof. We construct a suitable smooth function $\psi$ and estimate the constant $K$ with respect to the geometry of $\Omega$.
From each point $x\in\Sigma_1$, consider the ray $\gamma_x(t)=x+tN_x$, where $N_x$ is the exterior normal to $\Sigma_1$ at $x$ and $t\geq 0$.  
The  ray meets $\Sigma_2$ at a unique point $Q(x)$, and we let
$$
r(x)=d(x,Q(x)).
$$
Then we define $\psi$ as follows:
$$
\psi=\threesystem
{0\quad\text{on}\quad\Sigma_1}
{1\quad\text{on}\quad\Sigma_2}
{\text{linear on each ray from $\Sigma_1$ to $\Sigma_2$}}
$$
\begin{prop} \label{estimate doubly convex}Let
$$
\twosystem
{\beta=\min\{r(x): x\in\Sigma_1\}}
{B=\max\{r(x): x\in\Sigma_1\}}
$$
Then, at all points of $\Omega$ one has:
$$
\dfrac{1}{B}\leq\abs{\nabla\psi}\leq\dfrac{B}{\beta^2}.
$$
In particular:
$$
K_{\Omega,\psi}=\dfrac{\sup_{\Omega}\abs{\nabla\psi}}{\inf_{\Omega}\abs{\nabla\psi}}\leq\dfrac{B^2}{\beta^2}.
$$
\end{prop}

We remark that, as the shortest segment joining $\Sigma_1$ to $\Sigma_2$ meets both curves orthogonally, we also have:
$$
\beta=d(\Sigma_1,\Sigma_2).
$$

{\bf Proof of Theorem \ref{main2}}.  It is an  immediate consequence of the main estimate for cylinders and Proposition \ref{estimate doubly convex}. 
%%%%%%
\smallskip
\noindent
Then, it remains to prove  Proposition \ref{estimate doubly convex}: we state its main steps as follows.

\smallskip

{\bf Step 1.} {\it Given the ray $R_x$ joining $x\in\Sigma_1$ to $Q(x)\in\Sigma_2$, let $\theta_x$ be the angle between $R_x$ and the outer normal to $\Sigma_2$ at $Q(x)$. Then:
$$
\cos\theta_x\geq\dfrac{\beta}{B}.
$$}
 \smallskip
 
 {\bf Step 2.} {\it On the ray $R_x$ joining $x$ to $Q(x)$, consider the point $Q_t(x)$ at distance $t$ from $x$, and let $\theta_x(t)$  be the angle between $R_x$ and $\nabla\psi(Q_t(x))$. Then the function
$$
h(t)=\cos(\theta_x(t))
$$
is non-increasing in $t$. As $\theta_x(r(x))=\theta_x$ we have in particular:
$$
\cos(\theta_x(t))\geq \cos(\theta_x)\geq\dfrac{\beta}{B}
$$
for all $t\in [0,r(x)]$.}

\medskip

We will prove Step 1 and Step 2 below. 

\smallskip

{\bf Proof of Proposition \ref{estimate doubly convex}}. At any point of $\Omega$, let $\nabla^R\psi$ denote the radial part of $\nabla\psi$, which is the gradient of the restriction of $\psi$ to the ray passing through the given point. As such restriction is a linear function, one sees that
$$
\dfrac{1}{B}\leq\abs{\nabla^R\psi}\leq \dfrac{1}{\beta}.
$$
Since $\abs{\nabla\psi}\geq\abs{\nabla^R\psi}$ one gets immediately
$$
\abs{\nabla\psi}\geq\dfrac{1}{B}.
$$
Note that  $\theta_x(t)$, as defined in Step 2, is precisely the angle between $\nabla\psi$ and $\nabla^R\psi$, so that

$$
\abs{\nabla^R\psi}=\abs{\nabla\psi}\cos\theta_x(t),
$$
hence, by Step 2:
$$
\abs{\nabla^R\psi}\geq \frac{\beta}{B}
\abs{\nabla\psi}
$$
and then
$$
\abs{\nabla\psi}\leq \frac{B}{\beta^2}
$$
as asserted. 

\medskip

{\bf Proof of Step 1.} Let  $T_x$ be  the tangent line to $\Sigma_2$ at $Q(x)$ and $H(x)$  the point of $T_x$ closest to $x$. As $\Sigma_2$ is convex, $H(x)$ is outside $\Sigma_2$, hence
$$
d(x,H(x))\geq\beta.
$$
The triangle formed by $x, Q(x)$ and $H(x)$ is rectangle in $H(x)$, then we have:
$$
r(x)\cos\theta_x=d(x,H(x)).
$$
As $r(x)\leq B$ we conclude:
$$
B \cos\theta_x\geq \beta,
$$
which gives the assertion. \quad

\medskip

\textbf{Proof of Step 2.} We use a suitable parametrization of $\Omega$. Let us start from a parametrization
$\gamma:[0,L]\to \Sigma_1$ by arc-length $s$ with origin at a given point in $\Sigma_1$ and let $N(s)$ be the outer normal vector to $\Sigma_1$ at the point $\gamma(s)$.  Consider the set:
$$
\tilde\Omega=\{(t,s)\in [0,\infty)\times [0,L): t\leq \rho(s)\}
$$
where we have set $\rho(s)=r(\gamma(s))$. 
Introduce the diffeomorphism
$
\Phi:\tilde\Omega\to \Omega
$
defined by
$$
\Phi(t,s)=\gamma(s)+tN(s).
$$
Let us compute the metric tensor in the coordinates $(t,s)$. Write $\gamma'(s)=T(s)$ for the unit tangent vector to $\gamma$ and observe that, as $N(s)$ points outside $\Sigma_1$, one has $N'(s)=k(s)T(s)$, where $k(s)$ is the curvature of $\Sigma_1$, which, by the choices we made, is everywhere non-negative because $\Sigma_1$ is convex. Then:
$$
\twosystem
{d\Phi(\dfrac{\bd}{\bd t})=N(s)}
{d\Phi(\dfrac{\bd}{\bd s})=(1+tk(s))T(s)}
$$
If we set $\theta(t,s)=1+t k(s)$ the metric tensor is:
$$
g=\twomatrix{1}{0}{0}{\theta^2}
$$
and an orthonormal basis is then $(e_1,e_2)$, where
$$
e_1=\dfrac{\bd}{\bd t}, \quad e_2=\dfrac{1}{\theta}\dfrac{\bd}{\bd s}.
$$
In these coordinates, our function $\psi$ is written:
$$
\psi(t,s)=\dfrac{t}{\rho(s)}.
$$
Now
$$
\twosystem
{\scal{\nabla\psi}{e_1}=\derive{\psi}{t}=\dfrac{1}{\rho(s)}}
{\scal{\nabla\psi}{e_2}=\dfrac{1}{\theta}\derive{\psi}{s}=-\dfrac{t\rho'(s)}{\theta(t,s)\rho(s)^2}}.
$$
It follows that
$$
\abs{\nabla\psi}^2=\dfrac{1}{\rho^2}+\dfrac{t^2\rho'^2}{\theta^2\rho^4}=
\dfrac{\theta^2\rho^2+t^2\rho'^2}{\theta^2\rho^4}.
$$
Recall the radial gradient, which is the orthogonal projection of $\nabla\psi$ on the ray, whose direction is given by $e_1$. If we fix $x\in\Sigma_1$, we have
$$
\theta_x(t)=\text{angle between $\nabla\psi$ and $e_1$}
$$
and we have to study
$$
f(t)=\cos\theta_x(t)=\dfrac{\scal{\nabla\psi}{e_1}}{\abs{\nabla\psi}}=\dfrac{1}{\rho(s)\abs{\nabla\psi}}
$$
for a fixed $s$.  From the above expression of $\abs{\nabla\psi}$ and a suitable manipulation we see
$$
f(t)^2=\dfrac{\theta^2}{\theta^2+t^2g^2}
$$
where $g=\rho'(s)/\rho(s)$. Now
$$
\begin{aligned}
\dfrac{d}{dt}\dfrac{\theta^2}{\theta^2+t^2g^2}&=\dfrac{2t\theta g^2}{(\theta^2+t^2g^2)^2}
(t\derive{\theta}{t}-\theta)\\
\end{aligned}
$$
As $\theta(t,s)=1+tk(s)$ one sees that $t\derive{\theta}{t}-\theta=-1$ hence
$$
\dfrac{d}{dt}f(t)^2=-\dfrac{2t\theta g^2}{(\theta^2+t^2g^2)^2}\leq 0
$$
Hence $f(t)^2$ is non-increasing and, as $f(t)$ is positive, it is itself non-increasing. \qed

%%%%

\begin{ex} \label{example1} \rm We use the upper bound obtained in Proposition \ref{upperharmonic}, 
to show that in the lower bound of Theorem \ref{main2}, we have to take into account  the distance between the two boundaries.  
Fix  the rectangles :
$$
R_2 = [-4,4]\times [0,4], \quad R_{1,\epsilon}=[-3,3]\times [\epsilon,2]
$$ 
and consider the region $\Omega_{\epsilon}$ given by the closure of $R_2\setminus R_{1,\epsilon}$. Note that $\Omega_{\epsilon}$ is a planar annulus whose boundary components are convex and get closer and closer as $\epsilon\to 0$. 

\begin{figure}[h]\centering
\includegraphics[width=40mm]{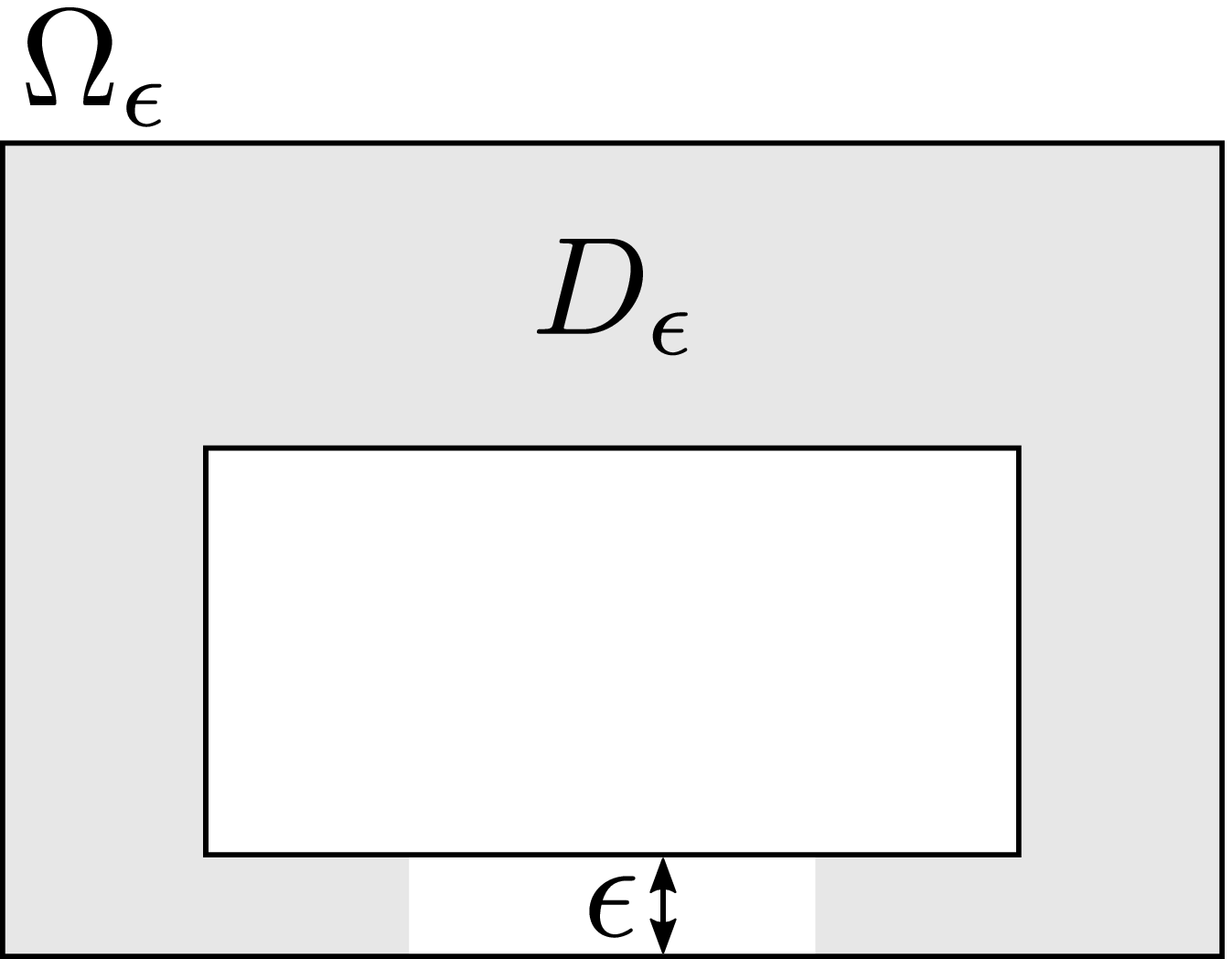}
\caption{$\lambda_1 \to 0$ as $\epsilon \to 0$}
\end{figure}

 We show that, for any closed potential $A$ one has:
\begin{equation}\label{small}
\lim_{\epsilon\to 0}\lambda_1(\Omega_{\epsilon},A)=0.
\end{equation}
Consider the simply connected region $D_{\epsilon}\subset\Omega_{\epsilon}$ given by the complement of the rectangle $[-1,1]\times [0,\epsilon]$. Now $D_{\epsilon}$ has trivial $1$-cohomology; by Proposition \ref{upperharmonic}, to show \eqref{small} it is enough to show that
\begin{equation}\label{enough}
\lim_{\epsilon\to 0}\nu_1(D_{\epsilon})=0.
\end{equation}
By the min-max principle :
$$
\nu_1(D_{\epsilon})=\inf\Big\{ \frac{\int_{D_{\epsilon}}\vert \nabla f\vert^2}{\int_{D_{\epsilon}}f^2} : f=0\,\,\text{on}\,\, \bd D_{\epsilon}^{\rm int} \Big\}
$$
where 
$$
\bd D_{\epsilon}^{\rm int} =\{(x,y)\in\Omega_{\epsilon}:x=\pm 1, y\in [0,\epsilon]\}.
$$
Define the test-function $f:D_{\epsilon}\to\reals$ as follows.
$$
f=\threesystem{1\quad\text{on the complement of $[-2,2]\times[0,\epsilon]$}}
{x-1\quad\text{on $[1,2]\times [0,\epsilon]$}}
{-x-1\quad\text{on $[-2,-1]\times [0,\epsilon]$}}
$$
One checks easily that, for all $\epsilon$:
$$
\int_{D_{\epsilon}}\vert \nabla f\vert^2=2\epsilon, \quad \int_{D_{\epsilon}}f^2\geq {\rm const}>0
$$
Then \eqref{enough} follows immediately by observing that the Rayleigh quotient of $f$ tends to $0$ as $\epsilon\to 0$.

\smallskip
			
\end{ex}

\begin{ex} \label{example2} \rm The following example  shows that we need to impose convexity in Theorem \ref{main2}. It is easy and classical fact to deform locally a domain in order to create a small eigenvalue for the Neumann problem as indicated in the figure. In our situation, up to a Gauge transformation, we can suppose that the potential $A$ is locally $0$ and we have to estimate the first eigenvalue of the Laplacian with Dirichlet boundary condition at the basis of the deformation and Neumann boundary condition on the remaining part of the mushroom, as explained in Proposition \ref{upperharmonic}.

\begin{figure}[h]\centering
\includegraphics[width=40mm]{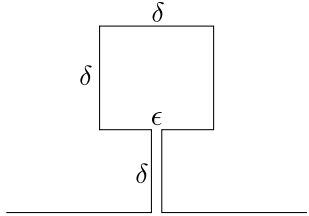}
\caption{A local deformation implying $\lambda_1 \to 0$}
\end{figure}

The only point is to take the value of the parameter $\epsilon$  much smaller than $\delta$ as $\delta \to 0$. Take $\epsilon =\delta^4$ and consider a function $u$ taking value $1$ in the square of size $\delta$ and passing linearly from $1$ to $0$ outside the rectangle of size $\epsilon,\delta$. The norm of the gradient of $u$ is $0$ on the square of size $\delta$ and $\frac{1}{\delta}$ in the rectangle of size $\delta,\epsilon$.

Then the Rayleigh quotient is
$$
R(u) \le \frac{\frac{1}{\delta^2}\delta \epsilon}{\delta^2}=\frac{\epsilon}{\delta^3} 
$$
which tends to $0$ as $\delta \to 0$.

\medskip
Moreover, we can do such local deformation keeping the curvature of the boundary bounded (see Example 2 in [CGI]).

\end{ex}

%%%%%%
%%%%%%%%%%%%%%

%%%%%%%

\section{Appendix}

\subsection{Spectrum of the circle} \label{riemannian circle}

We prove Theorem \ref{circle}.

Let then $(M,g)$ be the circle of length $L$ with metric $g=\theta(t)^2dt^2$, where $t\in [0,L]$ and $\theta(t)$ is periodic of period $L$. Given the $1$-form $A=H(t)dt$ we first want to find the harmonic $1$-form $\omega$ which is cohomologous to $A$; that is, we look for  a smooth function $\phi$ so that 
$
\omega=A+d\phi
$
is harmonic. Now a unit tangent vector field to the circle is
$$
 e_1=\dfrac{1}{\theta} \dfrac{d}{dt}.
$$ 
Write  $\omega=G(t)\,dt$. Then
$$
\delta\omega=-\dfrac 1{\theta}\Big(\dfrac{G}{\theta}\Big)'.
$$
As any $1$-form on the circle is closed, we see that $\omega$ is harmonic iff $G(t)=c\theta(t)$ for a constant $c$.
We look for  $\phi$ so that
$$
\phi'=-H+c\theta.
$$
Let us compute $c$. As $\phi$ must be periodic of period $L$, we must have $\int_0^L\phi'=0$. As the volume of $M$ is $L$, we also have $\int_0^L\theta=L$. Hence 
$$
c=\dfrac{1}{L}\int_0^LH(t)\,dt.
$$
On the other hand,  as the curve $\gamma(t)=t$ parametrizes $M$ with velocity $\frac{d}{dt}$,  one sees that the flux of $A$ across $M$ is given by
$$
\Phi^A=\dfrac{1}{2\pi}\int_0^LH
$$
Therefore
$
c=\frac{2\pi}{L} \Phi^A
$
and a primitive could be
$$
\phi(t)=-\int_0^tH+c\int_0^t\theta.
$$
Conclusion:

\nero {\it The form $A=H(t)dt$ is cohomologous to the harmonic form $\omega=c\theta\, dt$ with $c=\frac{2\pi}{L} \Phi^A$}.

\medskip

By Gauge invariance, we just need to compute the eigenvalues when the potential is $\omega$. In that case
$$
\Delta_{\omega}=-\nabla^{\omega}_{e_1}\nabla^{\omega}_{e_1}.
$$
Now
$$
\nabla^{\omega}_{e_1}u=\dfrac{u'}{\theta}-icu
$$
hence
$$
\nabla^{\omega}_{e_1}\nabla^{\omega}_{e_1}u=\dfrac{1}{\theta}\Big(\dfrac{u'}{\theta}-icu\Big)'-ic\Big(\dfrac{u'}{\theta}-icu\Big).
$$
After some calculation, the eigenfunction equation $\Delta_{\omega}u=\lambda u$ takes the form:
$$
-u''+\dfrac{\theta'}{\theta}u'+2ic \theta u'+c^2\theta^2 u=\lambda\theta^2 u.
$$
Set:
$$
u(t)=v(s(t)), \quad\text{that is}\quad v=u\circ s^{-1}.
$$
Then:
$$
\twosystem
{u'=v'(s)\theta}
{u''=v''(s)\theta^2+v'(s)\theta'}
$$
and the equation becomes:
$$
-v''+2ic v'+c^2v=\lambda v
$$
with solutions :
$$
v_k(s)=e^{\frac{2\pi i k}{L}s}, \quad \lambda=\dfrac{4\pi^2}{L^2}(k-\Phi^A)^2, \quad k\in\bf Z.
$$
Now Gauge invariance says that 
$$
\Delta_{A+d\phi}=e^{i\phi}\Delta_Ae^{-i\phi};
$$
and  $v_k$ is an eigenfunction of $\Delta_{A+d\phi}$ iff $e^{-i\phi}v_k$ is an eigenfunction of $\Delta_A$. Hence, the eigenfunctions of $\Delta_A$ (where $A=H(t)\,dt$) are
$$
u_k=e^{-i\phi}v_k,
$$
where $\phi(t)=-\int_0^tH+c\, s(t)$ and $c=\frac{2\pi}{L}\Phi^A$. Explicitly:
\begin{equation}\label{eigenfunctions}
u_k(t)=e^{i\int_0^tH}e^{\frac{2\pi i(k-\Phi^A)s(t)}{L}}
\end{equation}
as asserted. 

\smallskip

Let us know verify the last statement. If the metric is $g=dt^2$ then $\theta(t)=1$ and $s(t)=t$. If $A$ is a harmonic $1$-form then it has the expression $A=\frac{2\pi \Phi^A}{L}dt$. Taking into account \eqref{eigenfunctions} we see that $u_k(t)=e^{\frac{2\pi i k}{L}t}$.

\subsection{Spectrum of the flat rectangular torus}

\begin{prop}\label{torus} Fix positive numbers $p_1,\dots,p_n$ and let $M^n=\sphere 1(\frac{p_1}{2\pi})\times \cdots\times\sphere 1(\frac{p_n}{2\pi})$ be endowed with the product (flat) Riemannian structure. Let $A$ be a closed $1-$form on $M$. For each $\vec k=(k_1,\dots,k_n)\in {\bf Z}^n$ let
$$
\lambda_{\vec k}=\sum_{j=1}^n\dfrac{4\pi^2}{p_j^2}(k_j-\Phi^A_j)^2,
$$
where $\Phi^A_j$ is the flux of $A$ across the $j$-th factor
$\sphere 1(\frac{p_j}{2\pi})$.
Then the spectrum of the magnetic Laplacian with potential $A$ is given by $\{\lambda_{\vec k}:\vec k\in{\bf Z}^n\}$. In particular,
$$
\lambda_1(M,A)=\sum_{j=1}^n\dfrac{4\pi^2}{p_j^2}d(\Phi^A_j,{\bf Z})^2.
$$
\end{prop}

\begin{proof}
By Gauge invariance we can assume that $A$ is harmonic. Writing the metric in the canonical way, $g=\sum_{j=1}^ndt_j^2$,  and 
proceeding as in the case of the circle, we see that 
$$
A=A_1dt_1+\dots+A_ndt_n
$$
where $A_j=\frac{2\pi\Phi^A_j}{p_j}$.
 Let $(e_1,\dots,e_n)$ be the standard orthonormal basis 
$\Big(\opd{t_1},\dots,\opd{t_n}\Big)$. Then:
$$
\nabla^A_{e_j}u=\derive u{t_j}-iA_ju.
$$
Let $\vec k=(k_1,\dots,k_n)\in {\bf Z}^n$ and consider the function $u_{\vec k}: M\to\reals$ given by:
$$
u_{\vec k}(t_1,\dots,t_n)=e^{\frac{2\pi ik_1t_1}{p_1}}\cdots e^{\frac{2\pi ik_nt_n}{p_n}}
$$
Then for each $j$:
$$
\nabla^A_{e_j}u_{\vec k}=\dfrac{2\pi i}{p_j}(k_j-\Phi^A_j)u_{\vec k},
$$
from which we see that $u_{\vec k}$ is an eigenfunction of $\Delta_A$ associated to the eigenvalue
$$
\lambda_{\vec k}=\sum_{j=1}^n\dfrac{4\pi^2}{p_j^2}(k_j-\Phi^A_j)^2.
$$
As the collection of all $u_{\vec k}$, where $\vec k\in {\bf Z}^n$, is a complete orthonormal system in $L^2(M)$,  we see that the spectrum of $\Delta_A$ is 
$\{\lambda_{\vec k}:\vec k\in{\bf Z}^n$\}, as asserted. 
\end{proof}
%%%%%%%%%%%%%%

\subsection{Appendix: proof of Lemma \ref{calculus}} \label{technical lemma} For simplicity of notation, we give the proof when $a=L=1$. This will not affect generality. Then, assume that $s : [0,1]\times [0,1]\to\reals$ is smooth, non-negative and satisfies 
$$
s(0,t)=t,\quad s(r,0)=0,\quad s(r,1)=1\quad\text{and}\quad \derive st(r,t)\doteq \theta(r,t)>0.
$$
Assume the identity
\begin{equation}\label{identity}
F(t)=p(r)\cos(\pi s(r,t))+q(r)\sin(\pi s(r,t))
\end{equation}
for real-valued functions $F(t),p(r),q(r)$, such that $p(r)^2+q(r)^2$ is not identically zero.  Then we must show:
\begin{equation}\label{sr}
\derive sr=0
\end{equation}
everywhere. 
\medskip

Differentiate \eqref{identity} with respect to $t$ and get:
\begin{equation}\label{fprime}
F'(t)=-\pi p(r)\theta(r,t)\sin(\pi s)+\pi q(r)\theta(r,t)\cos(\pi s)
\end{equation}
and we have the following matrix identity
$$
\twomatrix{\cos(\pi s)}{\sin(\pi s)}{-\pi\theta\sin(\pi s)}{\pi\theta\cos(\pi s)}\Due{p}{q}=\Due{F}{F'}.
$$
We then see:
$$
p(r)=F(t)\cos(\pi s)-\dfrac{F'(t)}{\pi\theta}\sin(\pi s).
$$
Set $t=0$ so that $s=0$ and $p(r)=F(0)\doteq p$ is constant; the previous identity becomes 
\begin{equation}\label{p}
p=F(t)\cos(\pi s)-\dfrac{F'(t)}{\pi\theta}\sin(\pi s).
\end{equation}
Observe that:
\begin{equation}\label{changes}
\twosystem
{F'(0)=\pi q(r)\theta(r,0)}
{F'(1)=-\pi q(r)\theta(r,1)}
\end{equation}
\nero Assume $F'(0)=0$. Then, as $\theta(t,r)$ is positive one must have $q(r)=0$ for all $r$, hence $p\ne 0$ and 
$
F(t)=p\cos(\pi s),
$
from which, differentiating with respect to $r$, one gets easily
$
\derive{s}{r}=0
$
and we are finished. 

\nero We now assume that $F'(0)\ne 0$: then we see from \eqref{changes} that $q$ is not identically zero and the smooth function $F':[0,1]\to\reals$ changes sign. This implies that

\nero {\it  there exists $t_0\in (0,1)$ such that $F'(t_0)=0$.}

\smallskip

Now \eqref{p} evaluated at $t=t_0$ gives:
$$
p=F(t_0)\cos(\pi s(r,t_0))
$$
for all $r$. Differentiate w.r.t. $r$ and get, for all $r\in [0,1]$:
$$
0=\sin(\pi s(r,t_0))\derive{s}{r}(r,t_0).
$$ 
Since $s(r,t)$ is increasing in $t$, we have 
$$
0<s(r,t_0)<s(r,1)=1.
$$
Hence $\sin(\pi s(r,t_0))>0$ and we get
$$
\derive{s}{r}(r,t_0)=0.
$$
 \eqref{identity} writes:
$$
F(t)=p\cos(\pi s)+q(r)\sin(\pi s),
$$
and then, differentiating w.r.t. $r$:
$$
0=-p\pi \sin(\pi s)\derive sr+q'(r)\sin(\pi s)+\pi q(r)\cos(\pi s)\derive sr.
$$
Evaluating at $t=t_0$ we obtain
$
0=q'(r)\sin(\pi s(r,t_0))
$
which implies
$$
q'(r)=0
$$
hence $q(r)=q$, a constant. We conclude that
$$
F(t)=p\cos(\pi s)+q\sin(\pi s)
$$
for constants $p,q$.  We differentiate the above w.r.to $r$ and get:
$$
0=\Big(-\pi p \sin(\pi s)+\pi q\cos(\pi s)\Big)\derive sr
$$
for all $(r,t)\in [0,1]\times [0,1]$. Now,  the expression inside parenthesis is non-zero a.e. on the square. Then one must have $\derive sr=0$ everywhere and the final assertion follows.

\bigskip

\normalsize 
\noindent Bruno Colbois

\noindent Universit\'e de Neuch\^atel, Institut de Math\'ematiques \\
Rue Emile Argand 11\\
 CH-2000, Neuch\^atel, Suisse

\noindent bruno.colbois@unine.ch

\bigskip

\noindent Alessandro Savo

\noindent Dipartimento SBAI, Sezione di Matematica \\
Sapienza Universit\`a di Roma,  
Via Antonio Scarpa 16\\
00161 Roma, Italy

\noindent alessandro.savo@sbai.uniroma1.it

\end{document}